\documentclass[12pt]{article}
\usepackage{e-jc}
\usepackage[utf8]{inputenc}
\usepackage{xcolor,soul}
\usepackage{amsmath,enumerate,graphicx}
\usepackage{mathtools}
\usepackage{ytableau}
\usepackage{tikz,tikz-3dplot}

\usepackage{verbatim}
\usepackage{overpic}
\usepackage{geometry}

\DeclareMathOperator{\conv}{conv}
\DeclareMathOperator{\aff}{aff}
\DeclareMathOperator{\ehr}{ehr}
\DeclareMathOperator{\Ehr}{Ehr}
\DeclareMathOperator{\SSYT}{SSYT}
\newcommand{\Newt}{\mathrm{Newt}}
\newcommand{\polytope}{\mathcal{P}}
\newcommand{\Def}[1]{\textbf{#1}}
\newcommand{\R}{\mathbb{R}}
\newcommand{\Z}{\mathbb{Z}}
\newcommand{\Q}{\mathbb{Q}}
\def\a{\mathbf{a}}
\def\p{\mathbf{p}}
\def\u{\mathbf{u}}
\def\bv{\mathbf{v}}
\def\x{\mathbf{x}}
\def\y{\mathbf{y}}

\definecolor{green}{RGB}{34, 139, 34}

\definecolor{darkpink}{RGB}{245,75,200}

\definecolor{uw_purple}{RGB}{95,45,156} 

\title{Lattice polytopes from Schur and symmetric Grothendieck polynomials}
\author{Margaret Bayer\thanks{Partially supported by University of Kansas 
General Research Fund.}\\
\small Department of Mathematics \\[-0.8ex] 
\small University of Kansas \\[-0.8ex]
\small Lawrence, Kansas, U.S.A.  \\ 
\small\tt bayer@ku.edu \\ 
\and 
Bennet Goeckner\thanks{Partially supported by AMS-Simons Travel Grant.} \\ 
\small Department of Mathematics \\ [-0.8ex]
\small University of Washington \\[-0.8ex]
\small Seattle, Washington, U.S.A. \\ 
\small\tt goeckner@uw.edu\\
\and 
Su Ji Hong \\ 
\small Department of Mathematics \\ [-0.8ex]
\small University of Nebraska--Lincoln \\[-0.8ex]
\small Lincoln, Nebraska, U.S.A. \\ 
\small\tt sujihong@huskers.unl.edu 
\and 
Tyrrell McAllister \\ 
\small Department of Mathematics and Statistics \\[-0.8ex] 
\small University of Wyoming \\[-0.8ex]
\small  Laramie, Wyoming, U.S.A.  \\ 
\small\tt tmcallis@uwyo.edu\\
\and
McCabe Olsen \\
\small Department of Mathematics \\ [-0.8ex]
\small Rose-Hulman Institute of Technology \\[-0.8ex] 
\small Terre Haute, Indiana, U.S.A.  \\ 
\small\tt olsen@rose-hulman.edu \\
\and 
Casey Pinckney \\ 
\small Department of Mathematics \\ [-0.8ex]
\small Colorado State University \\[-0.8ex]
\small Fort Collins, Colorado, U.S.A. \\ 
\small\tt pinckney@math.colostate.edu\\
  \and 
Julianne Vega \\ 
\small Department of Mathematics \\ [-0.8ex]
\small Kennesaw State University \\[-0.8ex]
\small Kennesaw, Georgia, U.S.A. \\ 
\small\tt jvega30@kennesaw.edu\\ 
\and 
Martha Yip\thanks{Partially supported by Simons Collaboration Grant.} \\
\small Department of Mathematics \\ [-0.8ex]
\small University of Kentucky \\ [-0.8ex]
\small Lexington, Kentucky, U.S.A. \\ 
\small\tt martha.yip@uky.edu}

\date{\today}

\dateline{Jun 4, 2020}{Apr 26, 2021}{TBD}
\Copyright{The authors. Released under the CC BY-ND license (International 4.0).}

\begin{document}
\maketitle


\begin{abstract}

Given a family of lattice polytopes, two common questions in Ehrhart Theory are determining when a polytope has the integer decomposition property and 
determining when a polytope is reflexive.  
While these properties are of independent interest, the confluence of these properties is a source of active investigation due to conjectures regarding the unimodality of the $h^\ast$-polynomial.
In this paper, we consider the Newton polytopes arising from two families of polynomials in algebraic combinatorics: Schur polynomials and inflated symmetric Grothendieck polynomials. 
In both cases, we prove that these polytopes have the integer decomposition property by using the fact that both families of polynomials have saturated Newton polytope. 
Furthermore, in both cases, we provide a complete characterization of when these polytopes are reflexive. We conclude with some explicit formulas and unimodality implications of the $h^\ast$-vector in the case of Schur polynomials.

\end{abstract}


\section{Introduction}

Of central interest in
algebraic combinatorics are polynomials $f \in \mathbb{C}[x_1,x_2,\dots, x_m]$, which commonly appear as generating functions that encode some combinatorial information.
Associated to each polynomial $f$ is the Newton polytope $\Newt(f)$, which is the convex hull of the exponent vectors occurring in the monomials in $f$.
A polynomial has \Def{saturated Newton polytope} if every lattice point appearing in the Newton polytope corresponds to the exponent vector of a monomial in $f$ with nonzero coefficient~\cite{Monical-Tokcan-Yong}.

If a polynomial $f$ has saturated Newton polytope, then
checking if a monomial has nonzero coefficient is equivalent to checking if the corresponding integer lattice point is in the Newton polytope.
Adve, Robichaux and Yong \cite{Adve-Robichaux-Yong} use this perspective to
study the computational complexity of the ``nonvanishing problem'' for
polynomials, with a focus on Schubert polynomials.

In this paper, we study Newton polytopes arising from Schur polynomials and a generalization of symmetric Grothendieck polynomials, which we call inflated symmetric Grothendieck polynomials.  
We denote these polytopes by $\Newt(s_\lambda)$ and $\Newt(G_{h,\lambda})$, respectively.
The polytopes $\Newt(s_\lambda)$ and $\Newt(G_{1,\lambda})$ have previously been studied by Monical, Tokcan, and Yong~\cite{Monical-Tokcan-Yong} and Escobar and Yong~\cite{Escobar-Yong}, but many open questions remain. 
We are particularly interested in determining when $\Newt(s_\lambda)$ and $\Newt(G_{h,\lambda})$ are reflexive and when they have the integer decomposition property (IDP), both of which we define in Section~\ref{Section:Background}.

The Ehrhart series of a lattice polytope $\mathcal{P}$ is a combinatorial tool that enumerates the lattice points in dilations of $\mathcal{P}$.
The $h^\ast$-vector of $\mathcal{P}$, denoted $h^\ast(\mathcal{P})$, records the coefficients in the numerator of the rational function representing the Ehrhart series. 
Understanding the $h^\ast$-vectors of reflexive polytopes has been a topic of extensive recent research \cite{Braun-Unimodality-Survey}. Hibi showed that reflexive polytopes have palindromic $h^\ast$-vectors \cite{Hibi--Palindromic}.

A lattice polytope $\mathcal{P}$ is \Def{Gorenstein} if some positive integer dilate of $\mathcal{P}$ is reflexive, and hence the Gorenstein property
is a relaxation of reflexivity.
The following conjecture is commonly attributed to Ohsugi and Hibi \cite{Ohsugi-Hibi} in the modern literature, though it is a special case of conjectures of Brenti \cite{Brenti--LogConcave} and of Stanley \cite{Stanley--LogConcave}.

\begin{conjecture}[Ohsugi--Hibi \cite{Ohsugi-Hibi}] \label{conj-unimodal} If $\mathcal{P}$ is a Gorenstein polytope that has the integer decomposition property, then $h^\ast(\mathcal{P})$ is unimodal.
\end{conjecture}

A prominent open question in Ehrhart Theory is whether the $h^\ast$-vector of every lattice polytope with IDP is unimodal \cite{Schepers-VanLangenhoven};
this is related to a conjecture of Stanley's on the unimodality of $h$-vectors
of Cohen-Macaulay and Gorenstein domains~\cite{Stanley-Hilbert-Functions-CM}.
Schepers and Van Langenhoven~\cite{Schepers-VanLangenhoven} show that lattice parallelepipeds, which are among the simplest examples of polytopes that have IDP, also have unimodal $h^\ast$-vectors.
We show $\Newt(s_\lambda)$ and $\Newt(G_{h,\lambda})$ have IDP and then consider unimodality of the $h^\ast$-vector of $\Newt(s_\lambda)$. 

Our main contributions in this paper are as follows. We show that all Newton polytopes that arise from Schur polynomials have IDP, and characterize which of these are reflexive. 
We present closed-form expressions for the $h^\ast$-vectors of those that are reflexive and show that these vectors are all unimodal. 
This is a family for which Conjecture \ref{conj-unimodal} holds. 
Further, we consider symmetric Grothendieck polynomials, which are linear combinations of Schur polynomials and can be thought of as their inhomogeneous analogue. We show that all Newton polytopes arising from inflated symmetric Grothendieck polynomials have IDP, and characterize the very few that are reflexive.

\noindent \textbf{Acknowledgments:} The authors would like to thank Federico Castillo and Semin Yoo for helpful discussions and contributions to the early stages of the project.
We would also like to thank Michel Marcus and Avery St.\ Dizier for helpful 
comments on a previous version of the paper.

This work was completed in part at the 2019 Graduate Research Workshop in Combinatorics, which was supported in part by NSF grant \#1923238, NSA grant \#H98230-18-1-0017, a generous award from the Combinatorics Foundation, and Simons Foundation Collaboration Grants \#426971 (to M.~Ferrara) and \#315347 (to J.~Martin). 

\section{Background} \label{Section:Background}
In this section, we briefly recall notions from convex geometry, Ehrhart theory, and the study of Newton polytopes in algebraic combinatorics.

\subsection{Convex polytopes and Ehrhart theory}\label{Ehrhart}
A \Def{polytope} $\polytope\subset \R^m$ is the convex hull of finitely many points $\bv_1,\ldots,\bv_k\in\R^m$. That is,
	\[
	\polytope= \conv\{\bv_1,\ldots,\bv_k\}\coloneqq 
	\left\{ \x=\sum_{i=1}^k \nu_i \bv_i \ \bigg| \ 0\leq \nu_i\leq 1 \mbox{ and } \sum_{i=1}^k\nu_i=1\right\}.
	\] 
The inclusion-minimal set $V\subseteq \R^m$ such that $\polytope=\conv(V)$ is called the \Def{vertex set} of $\polytope$. 
A polytope is called \Def{lattice} (resp. \Def{rational}) if $\polytope=\conv(V)$ for $V\subseteq \Z^m$ (resp. $V\subseteq \Q^m$). 
Given a polytope $\polytope\subseteq\R^m$, the classical Minkowski--Weyl theorem states that we can express $\polytope$ as a bounded set of the  form
	\[
	\polytope=\{\x\in\R^m \mid \langle \a_i,\x \rangle \leq b_i \mbox{ for } i=1,\ldots,\ell\}
	\]
where $\langle \a_i,\x \rangle = \sum_{j=1}^m a_{ij}x_j$, for some $\a_1,\ldots,\a_\ell\in\R^m$ and $b_1,\ldots,b_\ell\in \R$.
If none of these constraints are redundant, each constraint defines, by 
equality, a \Def{facet} (i.e., codimension 1 face) of $\polytope$.
The \Def{dimension} of $\polytope$, denoted $\dim(\polytope)$, is defined to be the dimension of its affine span in $\R^m$.

Let $\polytope$ be a lattice polytope with $\dim(\polytope)=d\leq m$.
Given a positive integer $t$, let $t\polytope\coloneqq \{t\x \mid \x \in \polytope\}$ be the $t$-th dilate of $\polytope$.
The lattice point enumeration function
	\[
	\ehr_\polytope(t)\coloneqq \#(t\polytope\cap \Z^m)
	\] 
is called the \Def{Ehrhart polynomial} of $\polytope$.
By a classical result of Ehrhart \cite{Ehrhart}, this function agrees with a polynomial of degree $d$ in the variable $t$. 
Equivalently, one may also consider the \Def{Ehrhart series} of $\polytope$ which is defined to be the formal power series
	\[
	\Ehr_\polytope(z)\coloneqq 1+ \sum_{t\geq 1}\ehr_\polytope(t)z^t=\frac{1+h_1^\ast z+\cdots +h_{d-1}^\ast z^{d-1}+h_d^\ast z^d}{(1-z)^{d+1}}.
	\]
The numerator of the Ehrhart series is called the \Def{$h^\ast$-polynomial} and the vector of coefficients $h^\ast(\polytope)=(1,h^\ast_1,\ldots, h^\ast_d)$ the \Def{$h^\ast$-vector}.
By a result of Stanley \cite{Stanley-Hilbert-Functions-CM}, $(1,h^\ast_1,\ldots, h^\ast_d)\in \Z_{\geq 0}^{d+1}$. 
Studying $h^\ast(\polytope)$ often informs the algebraic and geometric structure of a lattice polytope $\polytope$. 

If $\mathbf{0}$ is in the interior of $\polytope\subset \R^m$, 
the \Def{(polar) dual polytope} of $\polytope$ is the 
polytope
	\[
	\polytope^\ast\coloneqq \left\{\y\in\R^m \mid \langle\y,\x\rangle\leq 1 \mbox{ for all } \x\in \polytope \right\}.
	\]
A polytope $\polytope$ with $\mathbf{0}$ in its interior is called 
\Def{reflexive} if $\polytope^\ast$ is a lattice polytope. 

Hibi (\cite{Hibi--Palindromic}) showed that $\polytope\subset \R^m$ with $\mathbf{0}$ in its interior is reflexive
if and only if for any facet $\polytope\cap\{\x\in\R^m \mid \langle\a,\x\rangle= b\}$ 
where $\a$ is primitive (meaning the greatest common divisor of the coordinates 
of $\a$ is 1) and $b>0$, we have $b = 1$. 
In this case, there are no lattice points between the hyperplane spanned by the
facet and its translation through $\mathbf{0}$, and
we say that $\mathbf{0}$ is \Def{lattice distance} 1 from the facet.

We will use the above characterization to extend the notion of reflexivity
to polytopes that are not full-dimensional or
for which the polytope has a nonzero lattice point in the relative interior.  
We say that a $k$-dimensional polytope $\polytope$ with a point $\p$ in its 
relative interior is \Def{reflexive} if there is a lattice-preserving 
linear transformation and translation by $-\p$ that takes $\polytope$
to a reflexive polytope in $\R^k$. 
This can be tested by checking that the lattice point $\p$ is lattice
distance 1 from all the facets of $\polytope$, that is, that there are no 
lattice
points in $\aff(\polytope)$ between the span of the facet and its translation
containing $\p$.  

\begin{theorem}[Hibi {\cite{Hibi--Palindromic}}]
Let $\mathcal{P}$ be a lattice polytope of dimension $d$ containing the origin in its interior and having Ehrhart series
\[
\textnormal{Ehr}_{\mathcal{P}}(z) = \frac{h_0^{\ast}+h_1^\ast z+\cdots +h_{d-1}^\ast z^{d-1}+h_d^\ast z^d}{(1-z)^{d+1}}.
\]
Then $\mathcal{P}$ is reflexive if and only if $h_i^{\ast}=h_{d-i}^{\ast}$ for all $0\leq i \leq \lfloor\frac{d}{2}\rfloor$.
\end{theorem}

In other words, the $h^\ast$-polynomial of a reflexive lattice polytope is a palindromic polynomial of degree $d$.

A relaxation of reflexivity is the Gorenstein property. 
We say that $\polytope$ is \Def{Gorenstein} if there is some positive integer $c$ such that $c\polytope$ is a reflexive polytope, and the integer $c$ is called the \Def{Gorenstein index} of $\polytope$. 
Similarly, this is completely detected by the Ehrhart series, as $\polytope$ is Gorenstein if and only if its $h^\ast$-polynomial is palindromic of degree ${d-c+1}$ by a result of De Negri and Hibi~\cite{DeNegri-Hibi}. 

Given a  lattice polytope $\polytope$, one can consider the interplay between the convex geometry of $\polytope$ in $\R^m$ with the induced  arithmetic structure of $\polytope\cap \Z^m$.
This motivates the discussion of triangulations and the integer decomposition property.
A \Def{(lattice) triangulation} $\mathcal{T}$ of $\polytope$ is a decomposition of $\polytope$ as a lattice simplicial complex.
We say that $\mathcal{T}$ is \Def{regular} if the triangulation is induced as the domains of linearity of a piecewise-linear, convex function $\sigma:\polytope\to \R$.
We say that $\mathcal{T}$ is \Def{unimodular} if each maximal simplex $\Delta\in \mathcal{T}$ is a \Def{unimodular simplex}, that is, if the vertices of $\Delta$ generate $\Z^d$. 
We say that $\polytope$ has the \Def{integer decomposition property (IDP)} if 
for any positive integer $t$ and any lattice point $\p \in t\mathcal{P}\cap \mathbb{Z}^m$,
there are $t$ lattice points $\bv_1, \ldots, \bv_t \in \mathcal{P}\cap \mathbb{Z}^m$ such that $\p= \bv_1+\cdots+\bv_t.$ 
The existence of a (regular) unimodular triangulation of $\polytope$ ensures that $\polytope$ has IDP.
This implication is strict, as one can construct examples of polytopes with IDP without a unimodular triangulation (see, e.g., \cite{Bruns-Gubeladze,Firla-Ziegler}).

A sequence $a_0,a_1,\dots,a_n$ of real numbers is \Def{unimodal} if there is some $0\leq j\leq n$ such that $a_0 \leq a_1 \leq \cdots \leq a_{j-1} \leq a_j \geq a_{j+1} \geq \cdots \geq a_n$.
A common investigatory theme in Ehrhart theory is determining under what conditions one may ensure that coefficients of the $h^\ast$-vector form a unimodal sequence. 
The most notable sufficient result is the following. 

\begin{theorem} \label{thm-GorensteinTriangulation} \emph{(Bruns and R\"omer \cite{Bruns-Roemer}, Athanasiadis \cite[Theorem 1.3]{Athanasiadis-h*-vectors}\footnote{In this paper, the author also acknowledges unpublished work of Hibi
and Stanley.})} 
If $\polytope$ is Gorenstein and admits a regular, unimodular triangulation, then $h^\ast(\polytope)$ is a unimodal sequence.
\end{theorem} 	
	
Given that these conditions are rather restrictive, it is natural to consider relaxations to determine if unimodality still holds. 
It is known that Gorenstein is not sufficient for unimodality as indicated by Payne~\cite{Payne}, though none of the examples from this reference exhibit IDP.

The following even broader question was posed by Scheppers and Van Langenhoven:

\begin{question}[Scheppers and Van Langenhoven \cite{Schepers-VanLangenhoven}]
If $\polytope$ has the integer decomposition property, is $h^\ast(\polytope)$ a unimodal sequence?

\end{question}

\subsection{Newton polytopes}
Given a polynomial $f=\sum_{\alpha}c_{\alpha} \x^{\alpha} \in \mathbb{C}[x_1,x_2,\dots, x_m]$ where $\alpha \in \mathbb{Z}^m_{\geq 0}$, the \Def{Newton polytope} $\Newt(f)$ of $f$ is defined as the convex hull of the exponent vectors of $f$.
That is, 
$$\Newt(f)\coloneqq \conv\{\alpha \mid c_{\alpha}\neq 0\}.$$
A polynomial $f$ has \Def{saturated Newton polytope (SNP)} if every lattice point  $\alpha\in\Newt(f)\cap \Z^m$ appears as an exponent vector of $f$, that is, $c_\alpha\neq 0$. 
This notion was introduced by Monical, Tokcan, and Yong in \cite{Monical-Tokcan-Yong}.

We now define our main objects of study in this paper, Newton polytopes arising from Schur polynomials and from inflated symmetric Grothendieck polynomials, both of which have SNP. 
A \Def{partition} of a nonnegative integer $n$ with at most $m$ parts is $\lambda=(\lambda_1,\lambda_2,\dots,\lambda_m)$ with $\lambda_1\geq \lambda_2 \geq \cdots \geq \lambda_m\geq 0$ and $\sum_{i=1}^m \lambda_i = n$. This is denoted by $\lambda\vdash n$.
The number of positive parts of $\lambda$ is denoted by $\ell(\lambda)$. 
The \Def{Young diagram} associated to $\lambda$ is an arrangement of boxes with $\lambda_i$ boxes in the $i$-th row, with rows aligned at the left.
Given partitions $\mu$ and $\lambda$ such that the Young diagram of $\lambda$ is contained in the Young diagram of $\mu$, the \Def{skew shape} $\mu/\lambda$ is the Young diagram consisting of boxes in $\mu$ which are not in $\lambda$.
A \Def{semistandard Young tableau} is a filling of a Young diagram with positive integers such that entries are weakly increasing along each row and strictly increasing along each column. 
Let $\SSYT^{[m]}(\mu/\lambda)$ denote the set of all semistandard Young tableaux of shape $\mu/\lambda$ with fillings from $[m]=\{1,\dots,m\}$.

\begin{definition}
Let $\x = (x_1,\ldots,x_m)$.
The \Def{Schur polynomial} in $m$ variables indexed by $\lambda\vdash n$ is 
\[
s_{\lambda}(\x)=\sum_{T\in\SSYT^{[m]}(\lambda)} \x^T,
\]
where $\x^T= x_1^{d_1(T)}\cdots x_m^{d_m(T)}$ such that $d_i(T)$ is the number of times $i$ appears in $T$. 
\end{definition}

\begin{example}\label{ex:(3)} 
Consider the partition $\lambda=(3,0,0)\vdash 3$. Let $m=3$ and $\x = (x_1,x_2,x_3)$. The semistandard Young tableaux are
$$\ytableausetup{smalltableaux}
\ytableaushort{111}\quad 
\ytableaushort{222}\quad 
\ytableaushort{333}\quad
\ytableaushort{112}\quad
\ytableaushort{113}\quad
\ytableaushort{122}\quad
\ytableaushort{133}\quad
\ytableaushort{223}\quad
\ytableaushort{233}\quad
\ytableaushort{123}      
$$
and the associated Schur polynomial is
\[
s_{(3,0,0)}(\x) = 
x_1^3+x_2^3+x_3^3
+ x_1^2x_2+x_1^2x_3+x_1x_2^2+x_1x_3^2+x_2^2x_3+x_2x_3^2  
+x_1x_2x_3.
\]
The Newton polytope $\Newt(s_{(3,0,0)}(\x))$ is the convex hull of the points \[
(3,0,0),(2,1,0),(2,0,1),(1,2,0),(1,1,1),(1,0,2),(0,3,0),(0,2,1),(0,1,2),(0,0,3)
.\]
\end{example}


\begin{example}\label{ex:(2,1)}
Consider the partition $\lambda=(2,1,0)\vdash 3$. Let $m=3$ and $\x = (x_1,x_2,x_3)$. The semistandard Young tableaux are 
$$
\ytableaushort{11,2}\quad
\ytableaushort{11,3}\quad
\ytableaushort{12,2}\quad
\ytableaushort{13,3}\quad
\ytableaushort{22,3}\quad
\ytableaushort{23,3}\quad
\ytableaushort{12,3}\quad
\ytableaushort{13,2}
$$
and the associated Schur polynomial is 
\[
s_{(2,1,0)}(\x) = 
x_1^2x_2+x_1^2x_3+x_1x_2^2+x_1x_3^2+x_2^2x_3+x_2x_3^2+2x_1x_2x_3.
\]
The Newton polytope $\Newt(s_{(2,1,0)}(\x))$ is the convex hull of the points
\[
(2,1,0),(2,0,1),(1,2,0),(1,0,2),(0,2,1),(0,1,2),(1,1,1).
\]
\end{example}

Since Schur polynomials are homogeneous polynomials, 
$\Newt(s_\lambda(x_1,\ldots, x_m))$ is an $(m-1)$-dimensional polytope in $\mathbb{R}^m$.
Consequently, the polytopes for the Schur polynomials in Example \ref{ex:(3)} and Example \ref{ex:(2,1)} are $2$-dimensional polytopes in $\mathbb{R}^3$. In Figure \ref{fig:schur}, we have depicted these polytopes (equivalently) in the plane for convenience. It is known that Schur polynomials have SNP~\cite[Proposition~2.5]{Monical-Tokcan-Yong}.

\begin{figure}[ht!]
\begin{center}
\includegraphics[height=3.5cm]{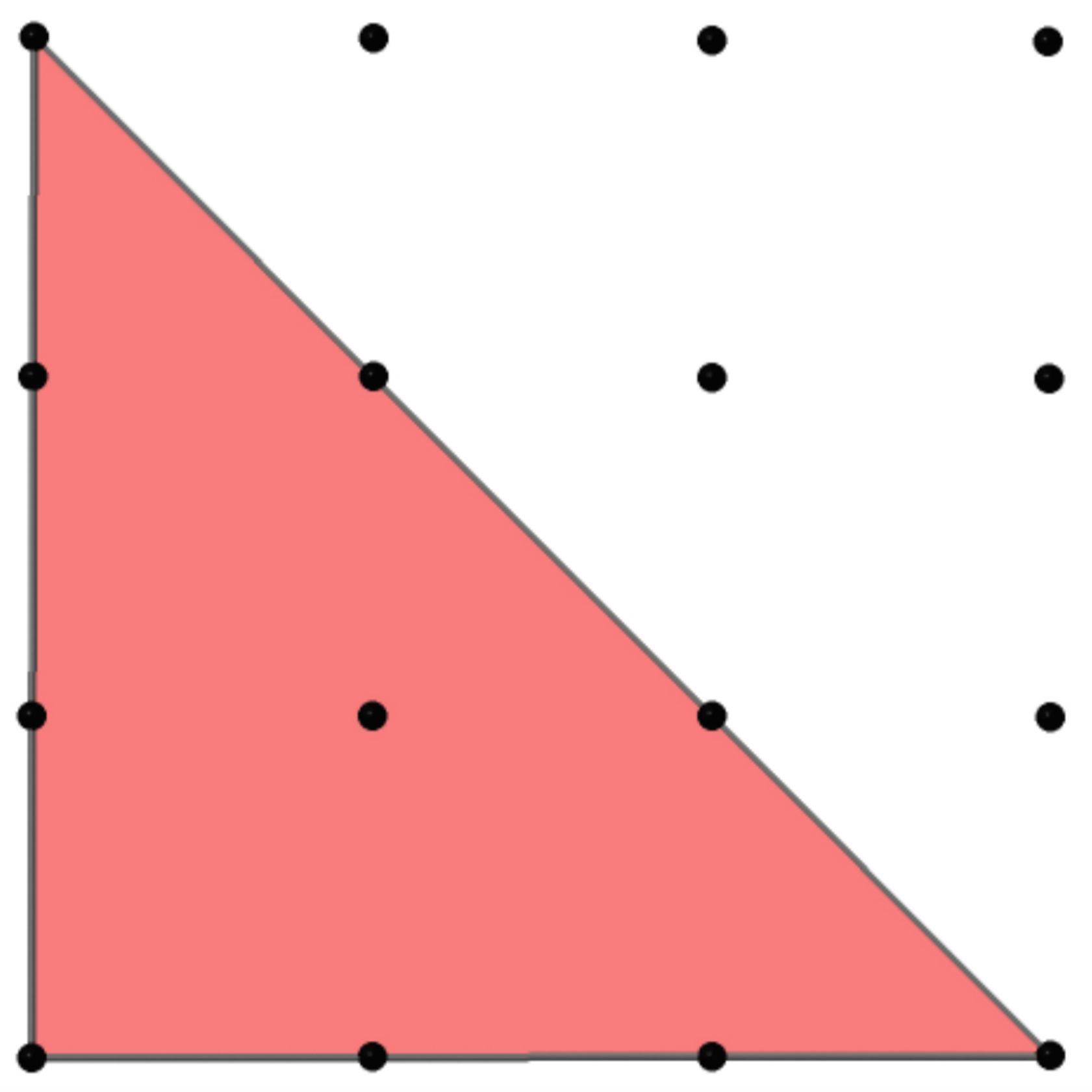} 
\hspace{2cm}
\includegraphics[height=3.5cm]{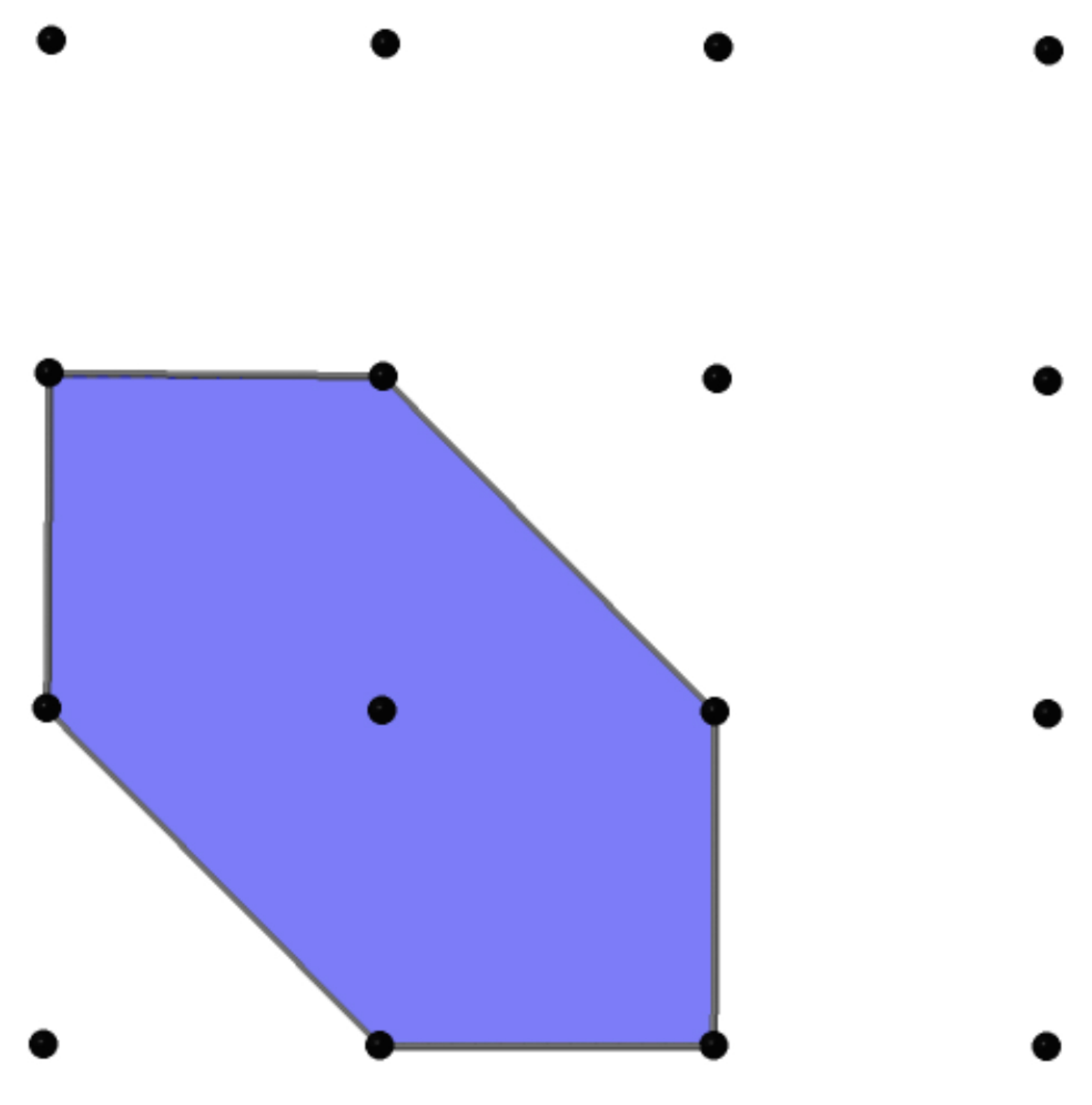} 
\caption{Newton polytopes for Schur polynomials in Example \ref{ex:(3)} (left) and Example \ref{ex:(2,1)} (right) drawn in $\mathbb{R}^2$ rather than in a $2$-dimensional subspace of $\mathbb{R}^3$. Both of these polytopes can be shown to be reflexive if we translate them so that their unique  interior point is $(0,0)$.}
\label{fig:schur}
\end{center}
\end{figure}

In Section~\ref{sec:Reflexive_Schur}, we characterize which Newton polytopes arising from Schur polynomials are reflexive. Figure \ref{fig:Reflexive-Nonreflexive-BG} illustrates examples of reflexive and nonreflexive $\Newt(s_\lambda(\x))$, generated using Normaliz \cite{Normaliz} and SageMath \cite{sagemath}.

\begin{figure}[ht!]
\begin{center}
\includegraphics[height=5.5cm]{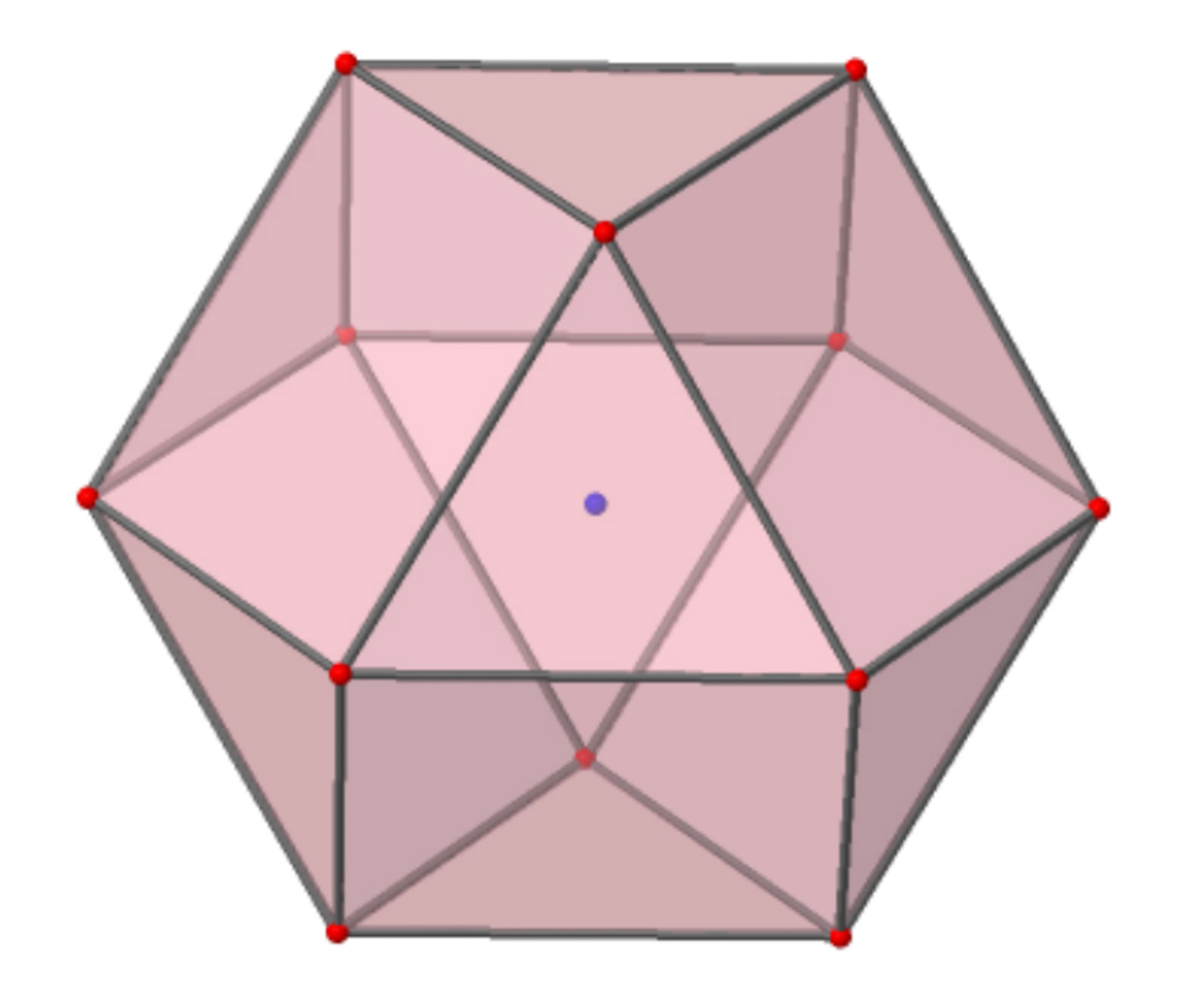} 
\hspace{1cm}
\includegraphics[height=5.5cm]{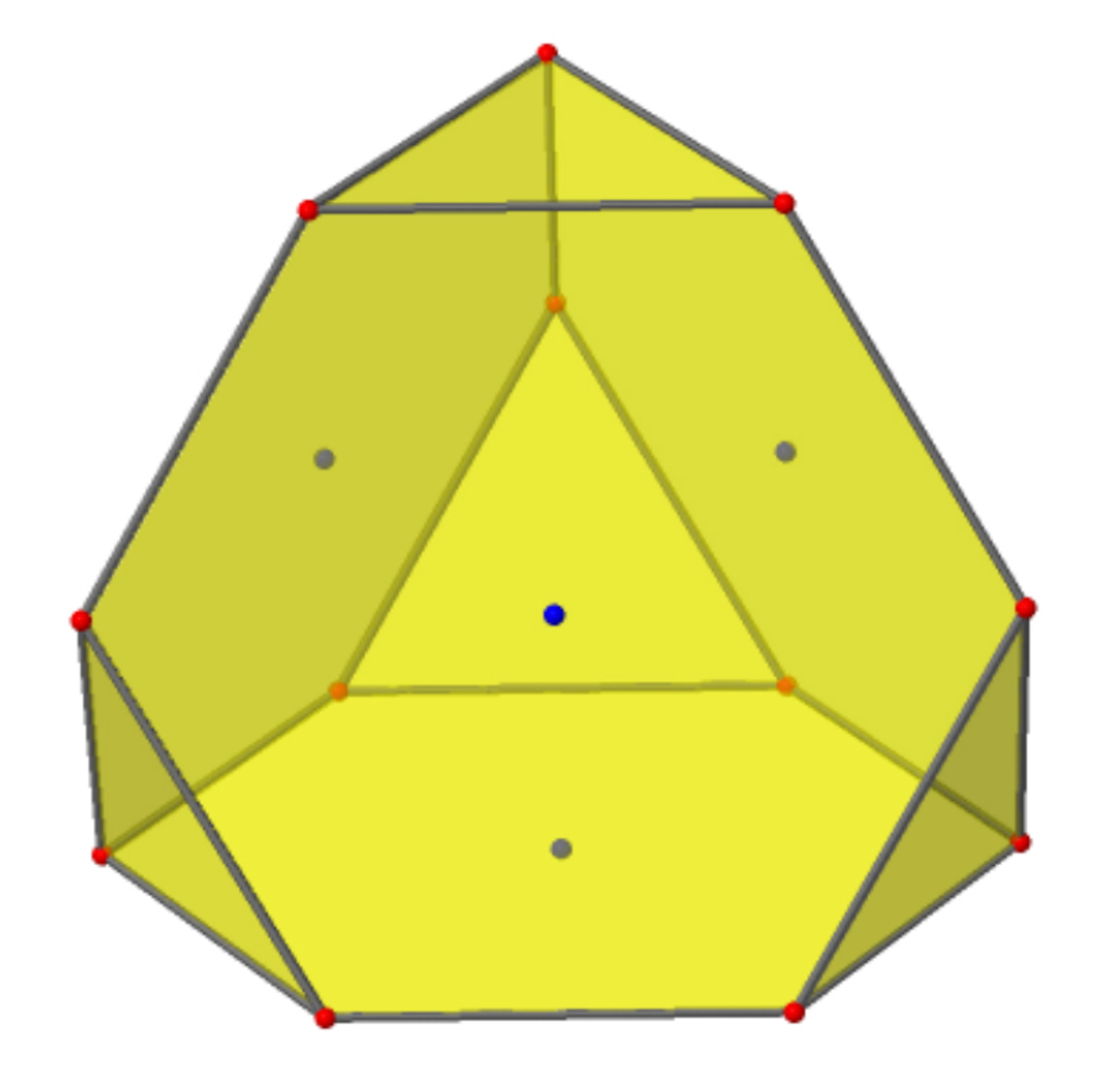} 
\caption{The polytope $\Newt(s_{(2,1,1,0)}(\x))$, on the left, is reflexive and the polytope $\Newt(s_{(2,1,0,0)}(\x))$, on the right, is not reflexive.}
\label{fig:Reflexive-Nonreflexive-BG}
\end{center}
\end{figure}

\begin{remark}
The Newton polytope $\Newt(s_\lambda(\x))$ of a Schur polynomial 
is the $(m-1)$-dimensional \Def{$\lambda$-permutohedron} $\mathcal{P}_\lambda^m$ in $\mathbb{R}^m$, which is the convex hull of the $S_m$-orbit of $(\lambda_1,\ldots,\lambda_m)\in \mathbb{R}^m$~\cite[Theorem 0.1]{Escobar-Yong}.
\end{remark}

Symmetric Grothendieck polynomials can be thought of as an inhomogeneous analogue of Schur polynomials. The following definition is due to Lenart~\cite[Theorem 2.2]{Lenart}.

\begin{definition} Let $\x=(x_1,\ldots, x_m)$ and let $\lambda$ be a partition with at most $m$ parts. For any partition $\mu\supseteq \lambda$ with at most $m$ rows, let $a_{\lambda\mu}$ be the number of fillings of the skew shape $\mu/ \lambda$ such that the filling increases strictly along each row and each column, and the filling in the $r$-th row is from $\{1,\ldots, r-1\}$. Let
$$A(\lambda) = \{ \mu \mid a_{\lambda\mu} \neq 0 \}. $$
The {\em symmetric Grothendieck polynomial} indexed by $\lambda$ is
$$G_\lambda(\x) = \sum_{\mu\in A(\lambda)} (-1)^{|\mu/\lambda|} a_{\lambda\mu} s_\mu(\x).$$
\end{definition}

\begin{example} \label{210example}
Let $\lambda =(2,1,0)\vdash 3$, $m=3$, and $\mathbf{x}=(x_1,x_2,x_3)$. Then
$$G_{(2,1,0)}(\mathbf{x}) = \textcolor{black}{s_{(2,1,0)}(\mathbf{x})} - \textcolor{black}{\left(s_{(2,2,0)}(\mathbf{x}) + 2s_{(2,1,1)}(\mathbf{x})\right)} + \textcolor{black}{2s_{(2,2,1)}(\mathbf{x})} - \textcolor{black}{s_{(2,2,2)}(\mathbf{x})}.$$
See Figure~\ref{fig.Grothendieck} for an illustration of the Newton polytope of $G_{(2,1,0)}(\mathbf{x})$. 
\end{example}

Escobar and Yong~\cite{Escobar-Yong} have shown that symmetric Grothendieck polynomials $G_\lambda(\x)$ have SNP.

\section{The Integer Decomposition Property}
In this section we will show that the Integer Decomposition Property (IDP) holds for Schur polynomials and a generalization of the symmetric Grothendieck polynomials.

\subsection{The Newton polytope of a Schur polynomial}
Using the realization of the Newton polytope $\Newt(s_\lambda(\x))$ as the $\lambda$-permutohedron $\mathcal{P}_\lambda^{m}$, we show that all Newton polytopes of Schur polynomials have IDP.
One should note that this result is already known by the theory of generalized permutohedra and polymatroids (see, e.g., {\cite[Corollary 46.2c]{Schrijver}}). 
However, we provide our proof as it motivates our methods for the Newton polytope of symmetric Grothendieck polynomials. 

\begin{proposition} \label{thm:Schur_IDP}
Let $\lambda$ be a partition with at most $m$ parts and let $\x = (x_1,\ldots, x_m)$. Then the Newton polytope $\Newt(s_\lambda(\x)) = \mathcal{P}_{\lambda}^m$ has the integer decomposition property.
\end{proposition}

\begin{proof}
The vertices of the $t$-th dilate $t\mathcal{P}_\lambda^{m}$ are the vertices of $\mathcal{P}_\lambda^{m}$ scaled by $t$, so the vertices of $t\mathcal{P}_\lambda^{m}$ are given by the $S_m$-orbit of $t\lambda$, and $t\mathcal{P}_\lambda^{m} = \Newt(s_{t\lambda}(\x))$.

Let $\p$ be a point in the $t$-th dilate $t\mathcal{P}_\lambda^m = \Newt(s_{t\lambda}(\x))$.  
Since $s_{t\lambda}(\x)$ has saturated Newton polytope, then $\p$ is the content vector of a semistandard Young tableaux $T$ of shape $t\lambda$.  
The tableau $T$ decomposes into $t$ semistandard Young tableaux $T_1,\ldots, T_t$ each of shape $\lambda$ such that $T_i$ consists of the $j$-th columns of $T$ for $j\equiv i\mod t$.
Letting $\bv_i$ denote the content vector of $T_i$, then $\p = \bv_1 + \cdots + \bv_t$, so $\Newt(s_\lambda(\x))$ has IDP.
\end{proof}

\begin{example} 
Let $m=3$, $\x=(x_1,x_2,x_3)$, and $\lambda = (2,1,0)\vdash 3$. 
Each lattice point in the dilated polytope $3\Newt(s_\lambda(\x)) = \Newt(s_{3\lambda}(\x))$ is the content vector of a semistandard Young tableau $T$ of shape $3\lambda = (6,3,0)$, and the lattice point can be decomposed into the sum of three points which are content vectors of semistandard Young tableaux $T_1, T_2, T_3$ of shape $\lambda$ by taking the  columns of $T$ mod $3$.

\definecolor{pink1}{HTML}{E18D96}
\definecolor{green2}{HTML}{C1CD97}
\definecolor{yellow3}{HTML}{F3DDB3}
\ytableausetup{boxsize=normal}
\begin{align*}
\begin{ytableau}
   *(pink1) 1 &*(green2) 1 &*(yellow) 2 &*(pink1) 2 &*(green2) 2 &*(yellow) 3 \\
   *(pink1) 2 &*(green2) 3 &*(yellow) 3 
  \end{ytableau}
&=
\begin{ytableau}
   *(pink1) 1 &*(pink1) 2 \\
   *(pink1) 2 
\end{ytableau}
+
\begin{ytableau}
   *(green2) 1&*(green2) 2 \\
   *(green2) 3
\end{ytableau}
+
\begin{ytableau}
   *(yellow) 2 &*(yellow) 3 \\
   *(yellow) 3 
\end{ytableau}\\ \\
(2,4,3) &= (1,2,0) + (1,1,1) + (0,1,2) 
\end{align*}
\end{example}

\subsection{The Newton polytope of a symmetric Grothendieck polynomial}

A notable difference between the Newton polytope of Schur polynomials versus symmetric Grothendieck polynomials is that unlike the case of Schur polynomials, $t\Newt(G_\lambda(\x)) \neq \Newt(G_{t\lambda}(\x))$. 
Motivated by our study of the integer decomposition property of the Newton polytope of symmetric Grothendieck polynomials, we make the following definition.

\begin{definition}
Let $h$ be a positive integer.  Let $\x=(x_1,\ldots, x_m)$ and let $\lambda \vdash n$ be a partition with at most $m$ parts. 
For any partition $\mu \supseteq \lambda$ with at most $m$ rows, let $b_{h,\lambda\mu}$ be the number of fillings of the skew shape $\mu/\lambda$ such that the filling increases strictly along each row and each column, and the filling in the $r$-th row is from $\{1,\ldots, h(r-1)\}$. 
Let
$$A(h,\lambda) = \{ \mu \mid b_{h,\lambda\mu}\neq0 \}. $$
The \Def{inflated symmetric Grothendieck polynomial} indexed by $\lambda$ and $h$ is
\[G_{h,\lambda}(\x) = \sum_{\mu\in A(h,\lambda)} (-1)^{|\mu/\lambda|} b_{h,\lambda\mu}s_\mu(\x).\]
\end{definition}

\begin{example} Let $\lambda =(2,1,0) \vdash 3$, $m=3$, $h=2$, and $\x=(x_1,x_2,x_3)$.  Then
$$G_{2,(2,1,0)}(\x) = s_{(2,1,0)}(\x) - (2s_{(2,2,0)}(\x) + 4s_{(2,1,1)}(\x))
   + 8s_{(2,2,1)}(\x) - 11 s_{(2,2,2)}(\x).$$ 
Compare with Example~\ref{210example}.
\end{example}

\begin{remark}
Note that $G_{1,\lambda}(\x)=G_\lambda(\x)$ is the usual symmetric Grothendieck polynomial. 
\end{remark}

Escobar and Yong~\cite{Escobar-Yong} showed that the symmetric Grothendieck polynomial $G_\lambda(\x)$ has SNP and described the components of the Newton polytope associated to the homogeneous components of $G_\lambda(\x)$. 
We extend the work of Escobar and Yong to $G_{h,\lambda}(\x)$ and show that $G_{h,\lambda}(\x)$ also has SNP.

\subsubsection{Inflated symmetric Grothendieck polynomials and SNP}
  
\begin{definition}
For two partitions $\mu,\lambda \vdash n$, we say $\mu$ \Def{dominates} $\lambda$ and write $\mu \,\unrhd\, \lambda$, if $\mu_1 + \cdots +\mu_i \geq \lambda_1 + \cdots + \lambda_i$ for every $i \geq 1$. 
\end{definition}

\begin{definition}\label{defn.partition_sequence} 
Let $h$ be a positive integer and let $\lambda$ be a partition with at most $m$ parts.
Let $\lambda^{(0)} =\lambda$ and for $k\geq 1$, let $\lambda^{(k)} \vdash |\lambda|+k$ be the partition obtained by adding a box to the $r_k$-th row of $\lambda^{(k-1)}$, where $r_k\in[m]$ is the smallest integer such that 
$$\lambda^{(k-1)}_{r_k} - \lambda_{r_k} < h(r_k-1),$$ and adding a box to the $r_k$-th row of $\lambda^{(k-1)}$ results in a valid partition. 
If $\deg G_{h,\lambda}(\x) = |\lambda|+N$, we say $\lambda^{(0)}, \ldots, \lambda^{(N)}$ is the \Def{sequence of dominating partitions for $G_{h,\lambda}(\x)$}.
\end{definition}

We justify this terminology with the next result. Lemma~\ref{lem.dominance}(a) is an extension of the result~\cite[Claim A]{Escobar-Yong} of Escobar-Yong to the case of inflated symmetric Grothendieck polynomials.

\begin{lemma} \label{lem.dominance} 
Let $\deg G_{h,\lambda}(\x)=|\lambda|+N$, and let $\{\lambda^{(0)}, \ldots, \lambda^{(N)}\}$ be the sequence of dominating partitions for $G_{h,\lambda}(\x)$.
\begin{enumerate}
\item[(a)] For $k=0,\ldots, N$, the partition $\lambda^{(k)}$ dominates all other partitions $\mu\in A(h,\lambda)$ such that $\mu \vdash |\lambda|+k$.
\item[(b)] The partition $\lambda^{(N)}$ is the unique partition of $ |\lambda|+N$ in $A(h,\lambda)$.
\item[(c)] $A(h,\lambda) = \{\mu \mid \lambda \subseteq \mu \subseteq \lambda^{(N)} \}. $
\end{enumerate}
\end{lemma}

\begin{proof}
Let $\mu\in A(h,\lambda)$ such that $\mu \vdash |\lambda|+k$.
Suppose for contradiction that $\lambda^{(k)}$ does not dominate $\mu$, so that there exists a minimum $s>1$ such that 
$\mu_1+\cdots+\mu_{s-1} \leq \lambda_1^{(k)} + \cdots +\lambda_{s-1}^{(k)}$
but
$\mu_1+\cdots+\mu_s > \lambda_1^{(k)} + \cdots +\lambda_s^{(k)}$. 
This implies $\mu_s > \lambda_s^{(k)}$. 

The partition $\lambda^{(k)}$ was obtained by adding a box to $\lambda^{(k-1)}$ in the $r_k$-th row.
If $s< r_k$, then
$$\lambda_s^{(k-1)}-\lambda_s \leq \lambda_s^{(k)}-\lambda_s < \mu_s - \lambda_s \leq h(s-1),$$
so a box would have been added to $\lambda^{(k-1)}$ in the $s$-th row to obtain $\lambda^{(k)}$, contradicting the construction of $\lambda^{(k)}$.  Thus $s\geq r_k$. 
But then by the construction of $\lambda^{(k)}$, for all $j\geq r_k$,
$$(\lambda_1^{(k)}+\cdots+\lambda_{j}^{(k)}) - \left(\lambda_1+\cdots+\lambda_{j} \right) = k 
\geq (\mu_1+\cdots+\mu_{j}) - \left(\lambda_1+\cdots+\lambda_{j} \right),$$
which contradicts the existence of $s$, so part (a) holds.

Parts (b) and (c) follow from the maximality of $\lambda^{(N)}$.
\end{proof}

\begin{proposition}
\label{proposition_layers_of_polytope}
Let $h$ be a positive integer, and let $\lambda$ be a partition with at most $m$ parts.  Suppose $\deg G_{h,\lambda}(\x) = |\lambda|+N$, and let $\lambda^{(0)},\ldots, \lambda^{(N)}$ be the sequence of dominating partitions for $G_{h,\lambda}(\x)$.
Further, let $H_k$ be the hyperplane in $\mathbb{R}^m$ defined by $\sum_{i=1}^m x_i = |\lambda|+k$. Then 
$$\Newt(G_{h,\lambda}(\x)) \cap H_k = \Newt(s_{\lambda^{(k)}}(\x)).$$
\end{proposition}

\begin{proof}
For $k\geq0$, if $(p_1,\ldots, p_m) \in H_k$, then $\sum_{i=1}^m p_i = |\lambda|+k$; thus $\Newt(G_{h,\lambda}(\x))\cap H_k$ is the convex hull of the content vectors of the partitions $\mu\in A(h,\lambda)$ such that $\mu \vdash |\lambda|+k$.

A result of Rado~\cite[Proposition 2.5]{Rado} states that 
$$\Newt(s_\alpha(\x))\subseteq \Newt(s_\beta(\x)) \hbox{ if and only if } \alpha \,\unlhd\, \beta$$ for any two partitions $\alpha, \beta$.
By Lemma~\ref{lem.dominance}(a), since $\lambda^{(k)}$ dominates all partitions $\mu\in A(h,\lambda)$ such that $\mu \vdash |\lambda|+k$, then $\Newt(s_\mu(\x)) \subseteq \Newt(s_{\lambda^{(k)}}(\x))$, and we conclude that $H_k \cap \Newt(G_{h,\lambda}(\x)) = \Newt(s_{\lambda^{(k)}}(\x))$. 
\end{proof}

\begin{remark}
The intersection of $\Newt(G_{h,\lambda}(\x))$ with the hyperplane $H_k$ corresponds to the homogeneous component of $G_{h,\lambda}(\x)$ of degree $|\lambda|+k$.
\end{remark}

\begin{proposition}
\label{proposition_snp_of_sgp}
The inflated symmetric Grothendieck polynomial $G_{h,\lambda}(\x)$ has SNP.
\begin{proof}
In~\cite{Escobar-Yong}, the proof that $G_{1,\lambda}(\x)=G_\lambda(\x)$ has SNP does not depend on the inflation parameter $h$ other than in~\cite[Claim A]{Escobar-Yong}, which describes the structure of $\Newt(G_\lambda(\x))$ arising from the homogeneous components of $G_\lambda(\x)$.  Using the description of the homogeneous components of $G_{h,\lambda}(\x)$ from Proposition~\ref{proposition_layers_of_polytope}, the rest of the proof in \cite{Escobar-Yong} applies to arbitrary $h\in \mathbb{Z}_{\geq1}$ and shows that $G_{h,\lambda}(\x)$ has SNP. 
\end{proof}
\end{proposition}

We revisit Example~\ref{210example} from the viewpoint of dominating partitions.  (Here $h=1$.)

\begin{example} \label{eg.sgp21}
Let $\lambda =(2,1,0)\vdash 3$, $m=3$, and $\x=(x_1,x_2,x_3)$. Then
$$G_{(2,1,0)}(\x) = \textcolor{blue}{s_{(2,1,0)}(\x)} - \textcolor{green}{\left(s_{(2,2,0)}(\x) + 2s_{(2,1,1)}(\x)\right)} + \textcolor{purple}{2s_{(2,2,1)}(\x)} - \textcolor{orange}{s_{(2,2,2)}(\x)}.$$
The sequence of dominating partitions for $G_{(2,1,0)}(\x)$ is
\ytableausetup{smalltableaux}
$$\textcolor{blue}{\lambda^{(0)}=\ydiagram{2,1}}\quad 
\textcolor{green}{\lambda^{(1)}=\ydiagram{2,2}}\quad 
\textcolor{purple}{\lambda^{(2)}=\ydiagram{2,2,1}}\quad 
\textcolor{orange}{\lambda^{(3)}=\ydiagram{2,2,2}}
$$
\end{example}


\begin{figure}[ht!] 
\hspace{1.7cm}
\begin{overpic}[height=5.35cm, angle = 360]{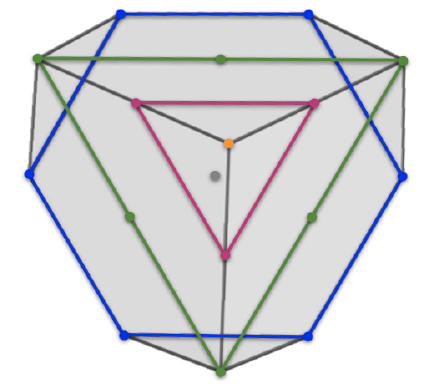}
	\put (100,3){\includegraphics[height=5.2cm]{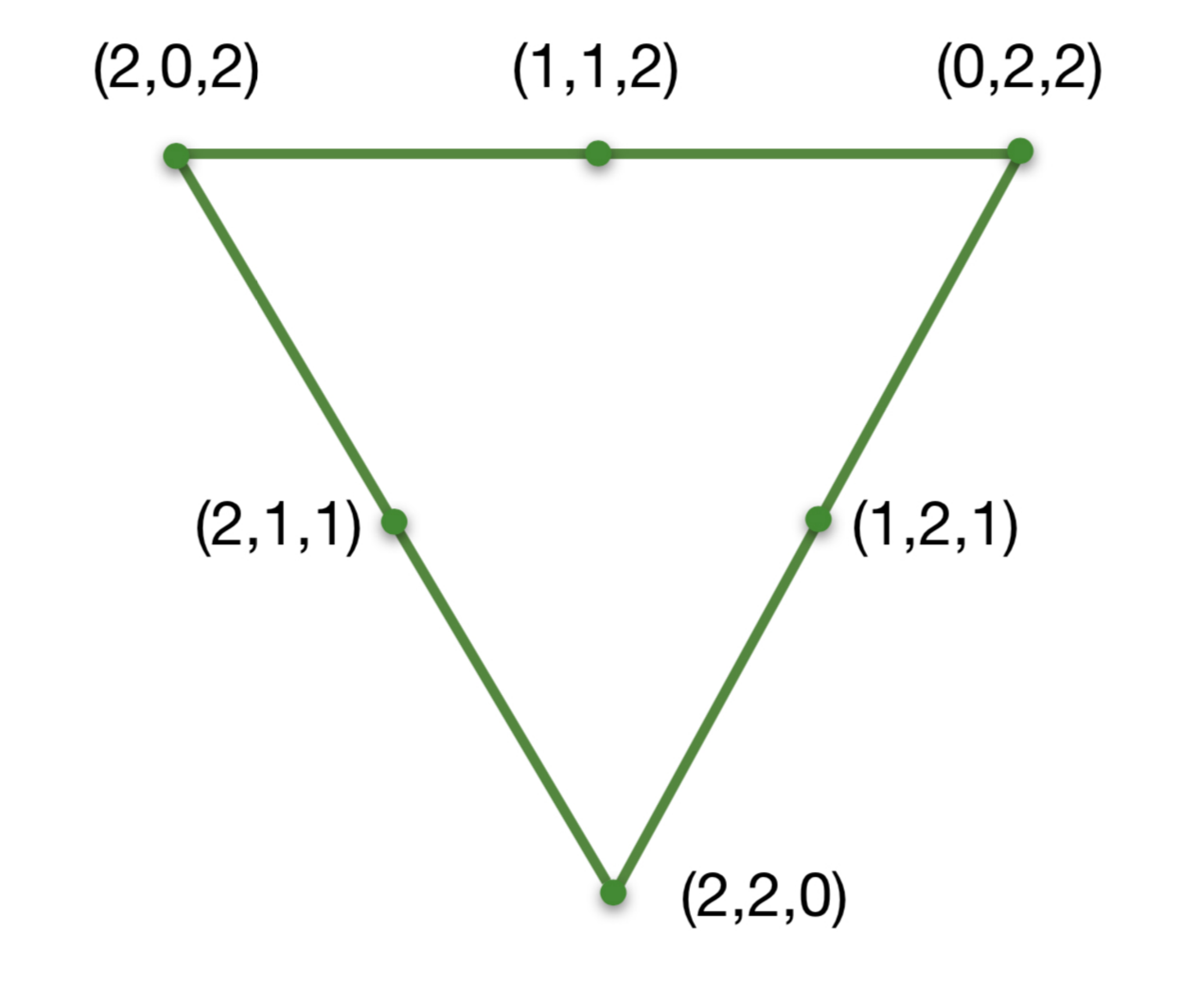}}
\end{overpic}
\caption{The lattice points of $\Newt(G_{(2,1,0)}(\mathbf{x}))$ are color-coded to reflect the structure arising from the four homogeneous components of $G_{(2,1,0)}(\mathbf{x})$; see Example~\ref{eg.sgp21}. The intersection of $\Newt(G_{(2,1,0)}(\mathbf{x}))$  with the hyperplane $x_1+x_2+x_3 =4$ is shown on the right. The extreme points are given by the $S_3$-orbit of the dominating partition $\lambda^{(1)}=(2,2,0)$.
}
\label{fig.Grothendieck}
\end{figure}

\subsubsection{Inflated symmetric Grothendieck polynomials and IDP}

We now show that the Newton polytopes of inflated symmetric Grothendieck polynomials have IDP. As a corollary, symmetric Grothendieck polynomials have IDP.

By Proposition~\ref{proposition_layers_of_polytope}, we know that $\Newt(G_{h,\lambda}(\x))$ is the convex hull of the $S_m$-orbits of the sequence of dominating partitions $\lambda^{(0)}, \ldots, \lambda^{(N)}$, but the next result shows that it suffices to take a certain subset of these partitions. To prove Proposition~\ref{proposition_vertices_of_sgp} we will need the following definition. 

\begin{definition}[Barvinok \cite{CourseInConvexity}]
Let $A \subseteq \mathbb{Z}^m$. A point $\mathbf{a} \in A$ is an \Def{extreme point} 
if $\mathbf{a} =  t\mathbf{b}+(1-t)\mathbf{c}$ for some $\mathbf{b}, \mathbf{c} \in A$ and $t\in (0,1)$ implies $\mathbf{b}=\mathbf{c}=\mathbf{a}$. 
\end{definition}
By the Minkowski--Weyl Theorem a polytope is the convex hull of the set of its 
extreme points (vertices). 
Thus, $\Newt(G_{h,\lambda}(\x))$ is the convex hull of its extreme points.


\begin{proposition}
\label{proposition_vertices_of_sgp}
Let $h$ be a positive integer, and let $\lambda$ be a partition with at most $m$ parts.  Suppose $\deg G_{h,\lambda}(\x) = |\lambda|+N$, and let $\lambda^{(0)},\ldots, \lambda^{(N)}$ be the sequence of dominating partitions for $G_{h,\lambda}(\x)$.
Suppose $\lambda^{(N)} =(\lambda_1+a_1, \ldots, \lambda_m+a_m)$
for some nonnegative integers $a_1,\ldots, a_m$, and let $b_k = a_1+\cdots+ a_k$ for $k=1,\ldots ,m$.
Then 
$$\Newt(G_{h,\lambda}(\x)) = \mathrm{conv} \bigcup_{k=1}^m \Newt(s_{\lambda^{(b_k)}}(\x))$$ 
is the convex hull of the $S_m$-orbits of the partitions $\lambda^{(b_1)}, \ldots, \lambda^{(b_m)}$.
Moreover, $\lambda^{(b_1)} = \lambda^{(0)} =\lambda$ and $\lambda^{(b_m)} = \lambda^{(N)}$.
\end{proposition}

\begin{proof}
We shall show that for $k=0,\ldots, N$, $\lambda^{(k)}$ is an extreme point of $\Newt(G_{h,\lambda}(\x))$ only if $k\in \{b_1,\ldots, b_m\}$.

If $\p=(p_1,\ldots,p_m)$ is an extreme point of $\Newt(G_{h,\lambda}(\x))$, then any permutation of $\p$ is also an extreme point of $\Newt(G_{h,\lambda}(\x))$, since $\Newt(G_{h,\lambda}(\x))$ is the convex hull of the $S_m$-orbits of $
\lambda^{(0)},\ldots,\lambda^{(N)}$ by Proposition~\ref{proposition_layers_of_polytope}. 

As $\lambda^{(N)} = (\lambda_1+a_1, \ldots, \lambda_m+a_m),$ then the largest number of boxes that can be added to the $r$-th row of $\lambda$ is $a_r$. 
Thus
$\lambda_r \leq \lambda^{(k)}_r \leq \lambda_r+a_r$
for each $k=0,\ldots, N$. By construction, 
$$\lambda^{(b_i)} = (\lambda_1 + a_1, \ldots, \lambda_i + a_i, \lambda_{i+1},\ldots, \lambda_m) $$
for each $i=1,\ldots,m$, so that each part of the partition is either at a maximum or a minimum.  
Thus if $\mu,\nu\in \Newt(G_{h,\lambda}(\x))$ are lattice points such that 
$t\mu+(1-t)\nu = \lambda^{(b_i)}$, then $\mu = \nu = \lambda^{(b_i)}$ necessarily.
So $\{\lambda^{(b_1)}, 
\ldots, \lambda^{(b_m)}\}$ is a set of extreme points of $\Newt(G_{h,\lambda}(\x))$.

On the other hand, suppose $k\notin \{b_1,\ldots, b_m\}$.  Then there exists $j$ such that
$$\lambda^{(k)} = (\lambda_1+a_1,\ldots, \lambda_{j-1}+a_{j-1}, \lambda_j + c , \lambda_{j+1},\ldots, \lambda_m) $$
with $0<c<a_j$. In this case, we have
\begin{align*}
\lambda^{(k-1)} 
	&= (\lambda_1+a_1,\ldots, \lambda_{j-1}+a_{j-1}, \lambda_j+c-1, \lambda_{j+1},\ldots, \lambda_m), \\
\lambda^{(k+1)} 
	&= (\lambda_1+a_1,\ldots, \lambda_{j-1}+a_{j-1}, \lambda_j+c+1, \lambda_{j+1},
\ldots, \lambda_m),
\end{align*}
so $\lambda^{(k)} = \frac12(\lambda^{(k-1)}+\lambda^{(k+1)})$.
Thus $\lambda^{(k)}$ is an extreme point of $\Newt(G_{h,\lambda}(\x))$ if and only if $k\in \{b_1,\ldots, b_m\}$.

Lastly, $a_1=0$ and $N=a_1+\cdots+a_m$, so $\lambda^{(b_1)} = \lambda^{(0)}=\lambda$ and $\lambda^{(b_m)} = \lambda^{(N)}$.
\end{proof}

\begin{example}\label{eg.210a}
Let $\lambda =(2,1,0)\vdash 3$ and $m=3$. 
Then $\Newt(G_{\lambda}(\x))$ is the convex hull of the $S_3$-orbit of 
$$\textcolor{blue}{\lambda^{(0)} = \ydiagram{2,1}}\quad
\textcolor{green}{\lambda^{(1)} = \ydiagram{2,2}} \quad\hbox{and}\quad
\textcolor{orange}{\lambda^{(3)} = \ydiagram{2,2,2}}$$
In Figure~\ref{fig.Grothendieck}, we see that that $\lambda^{(2)}=(2,2,1)$ is not an extreme point of the Newton polytope.
\end{example}

Recall that for a symmetric Grothendieck polynomial $G_{\lambda}(\x)$ we saw that $t\Newt(G_{\lambda}(\x))$ $\neq \Newt(G_{t\lambda}(\x)).$ Inflated symmetric Grothendieck polynomials are defined to address this discrepancy. 

\begin{proposition} \label{prop.dilate}
Let $t$ be a positive integer. Then
$$t\Newt(G_{h,\lambda}(\x)) = \Newt(G_{th,t\lambda}(\x)).$$
\end{proposition}
\begin{proof}
Let $\mathcal{P} = \Newt(G_{h,\lambda}(\x))$ and $\mathcal{Q} = \Newt(G_{th,t\lambda}(\x))$. 
Also let $\deg G_{h,\lambda}(\x)= |\lambda|+N$ while $\deg G_{th, t\lambda}(\x) = |t\lambda|+N'$.
By Proposition~\ref{proposition_vertices_of_sgp}, the vertices of $t\mathcal{P}$ and $\mathcal{Q}$ are determined by the partitions $t\lambda^{(N)}\vdash t|\lambda|+tN$ and $(t\lambda)^{(N')}\vdash t|\lambda|+N'$, respectively, so it suffices to show that $t\lambda^{(N)} = (t\lambda)^{(N')}$.  
Furthermore by Lemma~\ref{lem.dominance}(b), $(t\lambda)^{(N')}$ is the unique partition $\mu\in A(th,t\lambda)$ such that $b_{th,t\lambda\mu}\neq 0$ in the definition of the inflated symmetric Grothendieck polynomial $G_{th,t\lambda}(\x)$
and $|\mu|\vdash t|\lambda|+N'$, so it suffices to show that $t\lambda^{(N)} \vdash t|\lambda|+N'$.
Suppose
\begin{align*}
\lambda^{(N)} &= (\lambda_1+a_1,\ldots, \lambda_m+a_m),\\
(t\lambda)^{(N')} &= (t\lambda_1+a_1',\ldots, t\lambda_m+a_m').
\end{align*}
We shall show via induction that $ta_r = a_r'$ for $r=1,\ldots, m$. The base case is $r=1$, where by definition, $a_1' = 0 = a_1 = ta_1.$ Assume that $a_j' = ta_j$ for all $j<r$.

If $\lambda$ has a maximum of $a_r$ addable boxes in its $r$-th row, then at least $ta_r$ boxes can be added to the $r$-th row of $t\lambda$, and so $ta_r \leq a_r'$ for $r=1,\ldots, m$. 

Suppose $ta_r < a_r' \leq th(r-1)$.  Then $a_r < h(r-1)$ implies that the $(r-1)$-th and $r$-th rows of $\lambda^{(N)}$ have the same length.  In other words, $\lambda_{r-1}+a_{r-1} = \lambda_r + a_r$.  But by the induction hypothesis,
\begin{align*}
(t\lambda)_{r-1}^{(N')} 
	&= t\lambda_{r-1} +a_{r-1}'
	= t\lambda_{r-1} +ta_{r-1}
	= t\lambda_{r} +ta_{r}
	< t\lambda_r + a_r'
	= (t\lambda)_r^{(N)},
\end{align*}
which contradicts the fact that $(t\lambda)^{(N)}$ is a partition.
Therefore, $ta_r = a_r'$ for $r=1,\ldots,m$.  Since $t\lambda^{(N)} \vdash t|\lambda|+N'$, it follows from Lemma~\ref{lem.dominance}(b) that $t\lambda^{(N)}=(t\lambda)^{(N')}$.
The result follows.
\end{proof}

\begin{theorem}
Let $\lambda$ be a partition with at most $m$ parts and let $\x = (x_1,\ldots, x_m)$. Then the Newton polytope $\Newt(G_{h,\lambda}(\x))$ has the integer decomposition property.
\end{theorem}
\begin{proof}
Let $\mathcal{P} = \Newt(G_{h,\lambda}(\x))$ and $\mathcal{Q} = \Newt(G_{th,t\lambda}(\x))$.  
By Proposition~\ref{prop.dilate}, $t\mathcal{P} = \mathcal{Q}$.

Let $\nu = t\lambda$ and let $\p$ be a lattice point in the $t$-th dilate $t\mathcal{P} = \mathcal{Q}$.  The polynomial $G_{th,t\lambda}(\x)$ has SNP, and by Proposition~\ref{proposition_vertices_of_sgp}, the point $\p$ is a lattice point in $\Newt(s_{\nu^{(k)}})$ for some $k$, so $\p$ is the content vector of a semistandard Young tableau $T$ of shape $\nu^{(k)} \in A(th,t\lambda)$. 
The tableau $T$ decomposes into $t$ semistandard Young tableaux $T_1,\ldots, T_t$, where the tableau $T_i$ of shape $\theta(i)$ is obtained by taking the $j$-th columns of $T$ for $j \equiv i \mod t$.  We shall show that the partitions $\theta(1),\ldots, \theta(t)$ are in $A(h,\lambda)$.  
By Lemma~\ref{lem.dominance}(c), it suffices to show that $\lambda \subseteq \theta(i) \subseteq \lambda^{(N)}$, where $\deg G_{h,\lambda}(\x) = |\lambda|+N$.

The tableau $T_i$ is comprised of every $t$-th column of $T$, so its shape is
$$\theta(i)=(\lambda_1 + \ell_{i,1}, \ldots, \lambda_m + \ell_{i,m}) $$
for some nonnegative integers $\ell_{i,1},\ldots, \ell_{i,m}$.
Thus $\theta(i)\supseteq \lambda$.

Now, $\nu^{(k)}\in A(th,t\lambda)$, so Proposition~\ref{prop.dilate} gives
$$0\leq \nu_r^{(k)} - t\lambda_r  \leq ta_r,$$
for each $r=1,\ldots, m$, where $\lambda^{(N)} = (\lambda_1 +a_1,\ldots, \lambda_m+a_m)$.
Again, since the partition $\theta(i)$ is comprised of every $t$-th column $\nu^{(k)}$, then this implies
$$0 \leq \theta(i)_r - \lambda_r \leq a_r $$
for each $r=1,\ldots, m$.  Therefore, $\theta(i)\subseteq \lambda^{(N)}$.  So by Lemma~\ref{lem.dominance}(c), $\theta(i)\in A(h,\lambda)$.

Finally, let $k_i = \ell_{i,1}+\cdots+\ell_{i,m}$. Since $\theta(i)\vdash |\lambda|+k_i$ is a partition in $A(h,\lambda)$, then $\lambda^{(k_i)} \unrhd \theta(i)$.
If $\bv_i$ denotes the content vector of $T_i$, then $\bv_i$ is a lattice point in $\Newt(s_{\lambda^{(k_i)}}(\x))$, since $\Newt(s_{\lambda^{(k_i)}}(\x))$ has IDP.
As $\p = \bv_1 + \cdots + \bv_t$, then $\Newt(G_{h,\lambda}(\x))$ has IDP.
\end{proof}

\begin{corollary}
The Newton polytope $\Newt(G_\lambda(\x))$ has the integer decomposition property.
\end{corollary}

\section{Reflexivity}

\subsection{Reflexive and Gorenstein Newton polytopes of Schur polynomials} \label{sec:Reflexive_Schur}

We wish to characterize the Newton polytopes arising from Schur polynomials that
are reflexive.  
These polytopes are $(m-1)$-dimensional polytopes in $\R^m$;
they are contained in the hyperplane $\displaystyle \sum_{i=1}^m x_i = m$.
If we project onto the first $m-1$ coordinates, and then translate by
$(-1,-1,\dots, -1)$, we get an $(m-1)$-dimensional polytope in $\R^{m-1}$,
with $\mathbf{0}$ in its interior.
 So we can apply
the equivalent condition for reflexivity observed in Section~\ref{Ehrhart}
directly to the Newton polytope in $\R^m$.  That is, to show the
Newton polytope is reflexive we show there is a unique lattice point 
in the relative interior and it is lattice distance 1 from each facet.
Thus, to classify the reflexive $\lambda$-permutohedra,
 we first identify the facet-defining hyperplanes.

In this section we will often abuse terminology by using ``interior'' to
mean ``relative interior'' of a polytope, when the affine span of the 
polytope is clear.

Let $\lambda$ be a partition with at most $m$ parts and let $\x=(x_1,\ldots, x_m)$.
Recall that the Newton polytope $\Newt(s_\lambda(\x))$ is the $\lambda$-permutohedron $\mathcal{P}^m_\lambda$, which is the convex hull of the $S_m$-orbit of $(\lambda_1,\ldots, \lambda_m)$ in $\mathbb{R}^m$.
This polytope is of dimension $m-1$, and is determined by Rado's inequalities~\cite{Rado}:
$$\sum_{i\in I} x_i \leq  \sum_{i=1}^{|I|} \lambda_i \hbox{ for all }  I\subseteq [m], \qquad\hbox{ and }\qquad \sum_{i=1}^m x_i = |\lambda|.$$

Note that whether one of Rado's inequalities is facet-defining depends only on $|I|$ and $\lambda$, and not on the set $I$ itself.

\begin{theorem} \label{prop:facets}
The facets of $\mathcal{P}_\lambda^m$ are determined by the following inequalities.
\begin{enumerate}
\item[(a)] For all $i=1,\ldots, m$, $x_i\leq \lambda_1$, unless $\lambda_2 = \lambda_3 = \cdots =\lambda_m$.
\item[(b)] For $2\le |I|\le m-2$,
      $$\sum_{i\in I} x_i \leq \sum_{i=1}^{|I|}\lambda_i,$$ 
      unless
       $\lambda_1 = \lambda_2 = \cdots =\lambda_{|I|}$ or
       $\lambda_{|I|+1} = \lambda_{|I|+2} = \cdots =\lambda_{m}$. 
\item[(c)] For $|I|=m-1$,
      $$\sum_{i\in I} x_i \leq \sum_{i=1}^{m-1}\lambda_i,$$ 
      unless $\lambda_1 = \lambda_2 = \cdots =\lambda_{m-1}$.
\end{enumerate}
\end{theorem}
\begin{proof}
The set of permutations $(x_1,\ldots, x_m)$ of
$(\lambda_1, \ldots, \lambda_m)$
that satisfy
$\sum_{i\in I} x_i \leq \sum_{i=1}^{|I|} \lambda_i$ with equality is the set of products of the permutations of $(\lambda_1, \lambda_2,  \ldots, \lambda_{|I|})$
and the permutations of $(\lambda_{|I|+1}, \lambda_{|I|+2}, \cdots, \lambda_m)$.
The convex hull of the permutations of $(\lambda_1, \lambda_2,  \ldots, \lambda_{|I|})$ has dimension $0$ if $\lambda_1 = \lambda_2 = \cdots =\lambda_{|I|}$, and otherwise it has
dimension $|I|-1$.
Similarly, the convex hull of the permutations of $(\lambda_{|I|+1}, \lambda_{|I|+2}, \cdots, \lambda_m)$ has dimension 0 if $\lambda_{|I|+1} = \lambda_{|I|+2} = \cdots= \lambda_m$, and otherwise it has dimension $m-|I|-1$.
So the convex hull of the products has dimension $m-2$ exactly when $|I|=1$ and not all the $\lambda_i$ are equal for $i\geq2$, or when $|I|=m-1$ and not all the $\lambda_i$ are equal for $i\leq m-1$, or when $2\leq |I|\leq m-2$ and not all the $\lambda_i$ are equal for $i\leq |I|$ and not all the $\lambda_i$ are equal for $i\geq |I|+1$.
\end{proof}

Now that we know the facet-defining hyperplanes of 
$\Newt(s_\lambda(\x))=\mathcal{P}_\lambda^m$, we will 
use the lattice distance criterion to 
characterize which Schur polynomials give rise to reflexive Newton polytopes.

First we consider the following special case when $m<|\lambda|$.

\begin{proposition}\label{prop.mmm}
Let $m\geq2$, and $\lambda=(m,\ldots,m,0)\vdash m(m-1)$.  Then the Newton polytope $\Newt(s_\lambda(\x))=\mathcal{P}^m_\lambda$ is reflexive.
\end{proposition}
\begin{proof}
The $\lambda$-permutohedron $\mathcal{P}^m_\lambda$ is the convex hull of the $S_m$-orbit of $(m,\dots,m,0)$ in $\mathbb{R}^m$.  
The polytope lies in the hyperplane $\sum_{i=1}^m x_i = (m-1)m$, and by Theorem~\ref{prop:facets} the facets are determined by the inequalities $x_i\leq m$ for $i=1,\ldots, m$.  
Thus, the lattice point ${(m-1,\ldots, m-1)}$ in the interior of $\mathcal{P}^m_\lambda$ is unique, and it is lattice distance $1$ from the facets.  
Therefore, $\mathcal{P}^m_\lambda$ is reflexive.
\end{proof}


We next consider the case $m=|\lambda|$.

\begin{proposition}\label{refl-list}
Let $m\geq 2$.  The Newton polytope $\Newt(s_\lambda(\x)) = \mathcal{P}^m_\lambda$ is reflexive when $\lambda\vdash m$ is one of the following partitions:
\begin{enumerate}
\item[(a)] $\lambda=(m,0,\ldots,0)$,
\item[(b)] $\lambda=(2,\ldots,2,0,\ldots,0)$ when $m$ is even,
\item[(c)] $\lambda=(2,\ldots,2,1,0,\ldots,0)$ when $m$ is odd,
\item[(d)] $\lambda=(2,1,\ldots,1,0)$.
\end{enumerate}
\end{proposition}

\begin{proof}
We shall show in each case that $(1,\ldots,1)$ is the unique interior lattice point in $\mathcal{P}^m_\lambda$ and it is lattice distance one from every facet of the Newton polytope.

\begin{enumerate}
\item[(a)] The polytope $\mathcal{P}_\lambda^m$ is the convex hull of 
the $X_m$-orbit of $(m,0,\ldots, 0)$ in $\mathbb{R}^m$.  
By Theorem~\ref{prop:facets}, the facet-defining inequalities of $\mathcal{P}_\lambda^m$ are just those for $|I| =m-1$, and these can be written as 
$$ \sum_{i\neq j} x_i \leq m$$
for each $j=1,\ldots, m$.  
From this, we see that the point $(1,\ldots,1)$ is the unique lattice point in the interior of $\mathcal{P}_\lambda^m$, and it is lattice distance one from all facet-defining hyperplanes.  

Notice that $\mathcal{P}_\lambda^m$ lies in the hyperplane $\sum_{i=1}^m x_i =m$, so the inequality $ \sum_{i\neq j} x_i \leq m$ is equivalent to $x_j \geq0$.  
Although these two inequalities represent different half-spaces in $\mathbb{R}^m$, their intersection with the hyperplane $\sum_{i=1}^m x_i =m$ is the same.  
The point $(1,\ldots,1)$ is lattice distance one from both of the hyperplanes of $\mathbb{R}^m$ that bound these two half-spaces, and is lattice distance one from their intersection, considered as a hyperplane (of dimension $m-2$) in the $(m-1)$-dimensional space given by $\sum_{i=1}^m x_i =m$.

\item[(b)] Suppose $m\geq 4$ is even. 
The polytope $\mathcal{P}_\lambda^m$ is the convex hull of the $S_m$-orbit of $(2,\ldots, 2,0,\ldots, 0)$ in $\mathbb{R}^m$.
By Theorem~\ref{prop:facets}, the facet-defining inequalities of $\mathcal{P}_\lambda^m$ are $x_i\leq 2$ for each $i=1,\ldots,m$, and $\sum_{i\neq j} x_i \leq m$, for each $j=1,\ldots, m$.
Thus the point $(1,\ldots,1)$ is the unique lattice point in the interior of $\mathcal{P}_\lambda^m$, and it is lattice distance one from all facet-defining hyperplanes.

\item[(c)] Suppose $m\geq3$ is odd. 
This case is essentially the same as the previous.  
The facet-defining inequalities of $\mathcal{P}^m_\lambda$ are $x_i\leq 2$ and $x_i\geq 0$  for $i=1,\ldots, m$.
Thus the point $(1,\ldots,1)$ is the unique lattice point in the interior of $\mathcal{P}_\lambda^m$, and it is lattice distance one from all facet-defining hyperplanes.

\item[(d)] The polytope $\mathcal{P}^m_\lambda$ is the convex hull of the $S_m$-orbit of $(2,1,\ldots, 1,0)$ in $\mathbb{R}^m$.
By Theorem~\ref{prop:facets}, the facet-defining inequalities of $\mathcal{P}^m_\lambda$ are 
$$x_{i_1}+ \cdots + x_{i_s} \leq s+1$$
for all nonempty $I = \{i_1,\ldots, i_s\} \subseteq [m]$ with $|I|\leq m-1.$
Thus the point $(1,\ldots,1)$ is the unique lattice point in the interior of $\mathcal{P}_\lambda^m$, and is lattice distance $1$ from all facet-defining hyperplanes.  
\end{enumerate}
Therefore we conclude in each case that $\mathcal{P}_\lambda^m$ is reflexive.
\end{proof}

We next prove that Propositions~\ref{prop.mmm} and~\ref{refl-list} give a complete list of reflexive permutohedra.
To do this, we analyze the unique interior point of the reflexive polytope up to translation.  Let $\polytope^\circ$ denote the (relative) interior of the 
polytope $\polytope$.

\begin{lemma}
\label{prop:SingleInterior}
Let $m\geq2$ and let $\lambda\vdash n$ be a partition with at most $m$ parts. 
If $|(\mathcal{P}_\lambda^m)^\circ \cap \mathbb{Z}^m|=1$, then $m|n$.
\end{lemma}
\begin{proof}
If a lattice point is contained in the interior of $\mathcal{P}_\lambda^m$ then so is its entire $S_m$-orbit.
Thus, the only candidate for a single interior point is $(\frac{n}{m},\ldots,\frac{n}{m})$ which is only a lattice point when $m|n$.
\end{proof}

\begin{lemma}
\label{prop:SubtractFullColumns}
Let $m\geq2$ and let $\lambda\vdash n$ be a partition with $m$ parts. Let $\lambda'=(\lambda_1-\lambda_m,\ldots, \lambda_{m-1}-\lambda_m, 0)\vdash n-m\lambda_m$.
Then $\mathcal{P}_{\lambda'}^m$ is a translation of $\mathcal{P}_{\lambda}^{m}$.
\end{lemma}
\begin{proof}
The vertex description of $\lambda$-permutohedra implies that $\mathcal{P}_{\lambda'}^{m}$ is precisely the polytope $\mathcal{P}_\lambda^m$ translated linearly by the vector $(\lambda_m,\ldots, \lambda_m)$.
\end{proof}

In this case, we say that $\lambda$ $\Def{reduces by translation}$ to $\lambda'$.

\begin{theorem}
\label{thm:reflexive} Let $m\geq2$ and let $\lambda=(\lambda_1,\ldots,\lambda_m)\vdash n$ be a partition with at most $m$ parts.
The Newton polytope $\Newt(s_\lambda(\x))=\mathcal{P}_\lambda^m$ is reflexive if and only if $\lambda$ reduces by translation to $\lambda'$ of the following form:
\[\lambda' =
    \begin{cases}
    (m,0,\ldots,0) \vdash m,\\
    (2,1,\ldots,1,0) \vdash m,\\
    (2,\ldots,2,0,\ldots,0) \vdash m, & \hbox{ when $m$ is even},\\
    (2,\ldots,2,1,0,\ldots,0) \vdash m, & \hbox{ when $m$ is odd},\\
    (m,\ldots,m,0) \vdash m(m-1).
    \end{cases}
    \]
\end{theorem}

\begin{proof}
Propositions~\ref{prop.mmm} and~\ref{refl-list}, and Lemmas~\ref{prop:SingleInterior} and~\ref{prop:SubtractFullColumns} show that 
if $\lambda$ reduces to one of these forms, then $\mathcal{P}^m_{\lambda}$ is
reflexive.  We now prove the converse.

Suppose $\mathcal{P}_{\lambda}^m$ is reflexive. 
By Lemma~\ref{prop:SingleInterior} we may assume that $m|n$ and $\mathcal{P}_{\lambda}^m$ has the unique interior lattice point $(\frac{n}{m},\ldots,\frac{n}{m})$. 
By Lemma~\ref{prop:SubtractFullColumns}, if $\ell(\lambda)=m$, we may replace $\lambda$ by its translation by $(-\lambda_m,\dots,-\lambda_m)$.
Thus we assume that $\lambda_m=0$. 

We examine the cases $\lambda_2=0$ and $\lambda_2 >0$.
If $\lambda_2=0$, then $\lambda=(n,0,\ldots,0)\vdash n$. By Theorem~\ref{prop:facets}(c), $\sum_{i=1}^{m-1} x_i\leq n$ is a facet-defining hyperplane of $\mathcal{P}_{\lambda}^m$. Its unique interior lattice point is lattice distance one from this hyperplane, so $n - (m-1)\frac{n}{m} = 1$ implies $n=m$, giving the first case on the list of possible $\lambda'$.

Now assume $\lambda_2 > 0$.
We claim that $\lambda_1=\frac{n}{m}+1$.
First, since $\lambda \vdash n$ and $\lambda_m=0$, then $\lambda_1 \geq \frac{n}{m}+1$.
Suppose $\lambda_1=\frac{n}{m}+j$ for some positive integer $j$.
By Theorem~\ref{prop:facets}(a), for each $i\in[m]$, $\mathcal{P}^m_{\lambda}$ has the facet-defining hyperplane
    \[
    x_i\leq \frac{n}{m}+j.
    \]
The interior lattice point $(\frac{n}{m},\ldots,\frac{n}{m})$ is lattice distance $j$ from each of these facets. $\mathcal{P}_\lambda^m$ is reflexive implies $j=1$, so $\lambda_1 \leq \frac{n}{m}+1$. We conclude that $\lambda_1 = \frac{n}{m}+1$.

We next examine the subcases $\lambda_{m-1} = \frac{n}{m}+1$ and $\lambda_{m-1} \leq \frac{n}{m}$.
If $\lambda_{m-1}= \frac{n}{m}+1$, then $\lambda_i=\frac{n}{m}+1$ for all $i=1,\ldots, m-1$, so $n= \sum_{i=1}^{m} \lambda_i = (m-1)(\frac{n}{m}+1)$ implies $n=m(m-1)$ and $\lambda_i = \frac{m(m-1)}{m}+1 = m$ for $i=1,\ldots, m-1$.
This is the last case on the list of possible $\lambda'$.

Now assume $\lambda_{m-1} \leq \frac{n}{m}$.
Since $\lambda_1=\frac{n}{m}+1$, then by Theorem~\ref{prop:facets}(c), for any $I\subseteq[m]$ with $|I|=m-1$,
$\mathcal{P}_\lambda^m$ has the facet-defining hyperplane
    \[
    \sum_{i\in I} x_i\leq n.
    \]
The interior lattice point $(\frac{n}{m},\ldots,\frac{n}{m})$ is lattice 
distance $\frac{n}{m}$ from these facets, so $\mathcal{P}_\lambda^m$ reflexive implies $n=m$, and $\mathcal{P}_\lambda^m$ has the unique  interior lattice point $(1,\ldots,1)$. 

Continuing, we have $\lambda\vdash m=n$ with $\lambda_1=\frac{n}{m}+1 =2$, $\lambda_2 >0$, $\lambda_{m-1}\leq1$ and $\lambda_m=0$, so $\lambda$ is of the form $(2^k, 1^{m-2k}, 0^k)$ for some $k\geq1$.  Assuming that $m-2k\geq2$, then by Theorem~\ref{prop:facets}(b), for any subset $I\subseteq[m]$ with $|I|=k+1$,
$\mathcal{P}_\lambda^m$ has the facet-defining hyperplane
$$\sum_{i\in I} x_i\leq 2k+1.$$
The interior lattice point $(1,\ldots,1)$ is lattice distance $k$ from this facet, so $\mathcal{P}_\lambda^m$ is reflexive implies $k=1$, giving the second case on the list of possible $\lambda'$. 

Lastly, if $m-2k =0$ or $1$, this gives the remaining cases on the list of possible $\lambda'$.
Thus this completes the proof that $\mathcal{P}^m_\lambda$ is reflexive implies that $\lambda$ reduces to one of the $\lambda'$ on the list.
\end{proof}

The result of Theorem~\ref{thm:reflexive} allow us to give a characterization of the Gorenstein property as a corollary.

\begin{corollary} \label{cor:Schur_gorenstein}Let $\lambda$ be a partition with at most $m$ parts.
The Newton polytope $\mathcal{P}_\lambda^m = \Newt(s_\lambda(\x))$ is Gorenstein if and only if $\mathcal{P}_\lambda^m$ is reflexive or $\lambda$ reduces by translation to $\lambda'$ of the following form:
\[\lambda'=\begin{cases}
(k,0,\ldots,0),& \hbox{ where $k|m$},\\
(1^{m/2},0^{m/2}), & \hbox{ if $m$ is even},\\
(k,\ldots, k,0), & \hbox{ where $k|m$}.
\end{cases}\]
\end{corollary}
\begin{proof}
If $\mathcal{P}_\lambda^m$ is Gorenstein, there exists $a \in \mathbb{Z}_{\geq 1}$ so that $a\mathcal{P}_\lambda^m=\mathcal{P}_{a\lambda}^m$ is reflexive. 
If $a=1$, then $\mathcal{P}_\lambda^m$ is reflexive. 
Assume $\mathcal{P}_{a\lambda}^m$ is reflexive for some $a \geq 2$. 
Then $(a\lambda)'=a\lambda'$ is one of the cases in Theorem~\ref{thm:reflexive}.

Suppose $a\lambda'\vdash m$. If $a\lambda' = (m,0,\ldots,0)$, then $\lambda' = (\frac{m}{a},0,\ldots,0)$, so $\lambda_1'$ divides $m$.
If $a\lambda' = (2,\ldots,2,0,\ldots,0)$ where $m$ is even, then $\lambda' = (\frac{2}{a},\ldots, \frac{2}{a},0,\ldots,0)$.  As $a\geq2$, then $\lambda_1'=\frac{2}{a}=1$.

If $a\lambda' = (2,\ldots,2,1,0)$, then $\lambda' = (\frac{2}{a},\ldots,\frac{2}{a}, \frac{1}{a},0,\ldots,0)$. This is not possible as $a\geq2$.  Similarly, 
$a\lambda' \neq (2,1,\ldots,1,0,\ldots,0)$.

Lastly, suppose $a\lambda' = (m,\ldots, m,0)\vdash m(m-1)$. 
Then $\lambda'= (\frac{m}{a},\ldots,\frac{m}{a},0)$, so $\lambda_1',\ldots,\lambda_{m-1}'$ all divide $m$.
\end{proof}

As an immediate consequence of Corollary~\ref{cor:Schur_gorenstein}, we recover a result on the Gorenstein property for hypersimplices originally given by De Negri and Hibi~\cite[Theorem 2.4]{DeNegri-Hibi}.

\begin{corollary}
\label{cor:hypersimplices}
Let $\Delta_{k,n}$ be a hypersimplex.
Then $\Delta_{k,n}$ is Gorenstein if and only if $n=2k$, $k=1$, or $k=n-1$.
\end{corollary}

\subsection{Reflexive Newton polytopes of inflated symmetric Grothendieck polynomials}

We begin by determining the set of facet-defining hyperplanes of the Newton polytope of an inflated symmetric Grothendieck polynomial.  From this, we deduce which Newton polytopes are reflexive.

\begin{definition}
Let $f$ be a linear functional and let $H = \{\x \in \mathbb{R}^m \mid f(\x) = a\}$ be an affine hyperplane in $\mathbb{R}^m$.  Define the closed half-spaces
\begin{align*} \overline{H_+} &= \{\x \in \mathbb{R}^m \mid f(\x)\geq a\},\\
\overline{H_-} &= \{\x \in \mathbb{R}^m \mid f(\x)\leq a\}.
\end{align*}
For any set $S\subseteq \mathbb{R}^m$, we say $H$ \Def{isolates} $S$ if $S$ is contained in $\overline{H_+}$ or $\overline{H_-}$.
\end{definition}
Given an $m$-dimensional polytope $\polytope$, if $H$ is an affine hyperplane which isolates $\polytope$ and $\dim (\polytope \cap H)=m-1$, then $H$ is a facet-defining hyperplane of $\polytope$.

For the remainder of this section, we assume that $\lambda$ is a partition with at most $m$ parts and that it is reduced by translation, so $\lambda= (\lambda_1,\ldots, \lambda_{m-1},0)$.
Also let $\deg G_{h,\lambda}(\x) = |\lambda|+N$, so the sequence of dominating partitions is $\lambda^{(0)},\ldots, \lambda^{(N)}$.  

We pinpoint some facet-defining inequalities of $\Newt(G_{h,\lambda}(\x))$ to show that if it is reflexive, then there is a very limited region where its unique interior lattice point may lie.

\begin{lemma} \label{lem.frontandback}
The inequality $x_1+\cdots+x_m \geq |\lambda|$ is a facet-defining inequality of the polytope $\Newt(G_{h,\lambda}(\x))$. Furthermore, if the $S_m$-orbit of $\lambda^{(N)}$ is non-trivial, then $x_1+\cdots+x_m \leq |\lambda|+N$ is also a facet-defining inequality of $\Newt(G_{h,\lambda}(\x))$.
\end{lemma}
\begin{proof}
Recall that the Newton polytope $\Newt(s_{\lambda^{(k)}}(\x))$ lies in the hyperplane $x_1+\cdots +x_m = |\lambda^{(k)}| = |\lambda|+k$, and from Proposition~\ref{proposition_layers_of_polytope} we know that 
$$\Newt(G_{h,\lambda}(\x)) = \conv \left(\coprod_{k=0}^N \Newt(s_{\lambda^{(k)}}(\x)) \right),$$
so the hyperplane $H_{|\lambda|}$ defined by $x_1+\cdots+x_m=|\lambda|$ isolates $\Newt(G_{h,\lambda}(\x))$. 
Also, $\Newt(G_{h,\lambda}(\x)) \cap H_{|\lambda|} = \Newt(s_\lambda(\x))$, which has dimension $m-1$ because we assume that $\lambda$ is reduced by translation and thus does not have a trivial $S_m$-orbit. Thus $x_1+\cdots+x_m=|\lambda|$ is a facet-defining hyperplane of $\Newt(G_{h,\lambda}(\x))$. 

Similarly, if the $S_m$-orbit of $\lambda^{(N)}\in \mathbb{R}^m$ is non-trivial, then $\Newt(G_{h,\lambda}(\x)) \cap H_{|\lambda|+N} = \Newt(s_{\lambda^{(N)}}(\x))$ is $(m-1)$-dimensional. Since $\Newt(s_{\lambda^{(N)}}(\x))$ lies in the hyperplane $x_1+\cdots+x_m=|\lambda|+N$, then it is a facet-defining hyperplane of $\Newt(G_{h,\lambda}(\x))$. 
The result follows as $\Newt(G_{h,\lambda}(\x))$ lies in $\overline{H_{|\lambda|}}_+$ and $\overline{H_{|\lambda^{(N)}|}}_-$.
\end{proof}

\begin{corollary}\label{cor.uonone}
If $\Newt(G_{h,\lambda}(\x))$ is reflexive, then its unique interior lattice point must lie on $\Newt(s_{\lambda^{(1)}}(\x))$. 
\end{corollary}
\begin{proof}
Let $\u=(u,\ldots,u)$ be the unique interior lattice point of the reflexive polytope. It must lie on $\Newt(s_{\lambda^{(k)}}(\x))$ for some $k=1,\ldots, N-1$, so $mu = |\lambda|+k$.  Furthermore, by Lemma~\ref{lem.frontandback}, $x_1+\cdots+x_m=|\lambda|$ is a facet-defining hyperplane of $\Newt(G_{h,\lambda}(\x))$ and because the polytope is reflexive, $\u$ is lattice distance one from this hyperplane; therefore we can conclude that $k=1$.
\end{proof}

Recall from Proposition~\ref{proposition_vertices_of_sgp} that $\Newt(G_{h,\lambda}(\x))$ is the convex hull of the union of $\Newt(s_{\lambda^{(b_k)}}(\x))$ for $k=1,\ldots, m$, where if $\lambda^{(N)}= (\lambda_1+a_1,\ldots, \lambda_m+a_m)$ then $b_k=a_1+\cdots+a_k$. 
We next show how the remaining facet-defining inequalities of $\Newt(G_{h,\lambda}(\x))$ arise from the facet-defining inequalities of the polytopes $\Newt(s_{\lambda^{(b_k)}}(\x))$.

Let $F=F(I)$ be a facet of $\Newt(s_{\lambda^{(b_k)}}(\x))$ defined by the inequality
\[
	\sum_{i\in I} x_i \leq \sum_{i=1}^{|I|} \lambda_i^{(b_k)} \hbox{ for some proper nonempty } I\subset [m].
\]   
(Recall from Theorem~\ref{prop:facets} that all facet-defining inequalities of $\Newt(s_{\lambda^{(b_k)}}(\x))$ are of this form.)
Since $\Newt(s_{\lambda^{(b_k)}}(\x))$ lies in the affine hyperplane $x_1+\cdots + x_m = |\lambda^{(b_k)}|$, then $F$ is of the form
$$F= H(F)_{1} \cap \Newt(s_{\lambda^{(b_k)}}(\x)) 
= H(F)_{0} \cap \Newt(s_{\lambda^{(b_k)}}(\x)),$$
where $L=\sum_{i=1}^{|I|} \lambda_i^{(b_k)}$ and the affine hyperplanes are
given by
$$
\begin{array}{ll}
H(F)_{1} &= \left\{ \x \in \mathbb{R}^m \mid \hbox{$\sum_{i\in I}x_i = \sum_{i=1}^{|I|} \lambda_i^{(b_k)}$} = L\right\},\\
H(F)_{0} &= \left\{ \x \in \mathbb{R}^m \mid \hbox{$\sum_{i\notin I}x_i = \sum_{i=|I|+1}^{m} \lambda_i^{(b_k)}$} = |\lambda^{(b_k)}|-L\right\}.
\end{array}
$$
We will determine which of these hyperplanes isolate $\Newt(G_{h,\lambda}(\x))$.  

\begin{lemma}\label{lem.facetsofG} 
Let $I$ be a proper nonempty subset of $[m]$, let $F=F(I)$ be a facet of $\Newt(s_{\lambda^{(b_k)}}(\x))$ defined by 
$\sum_{i\in I} x_i \leq \sum_{i=1}^{|I|} \lambda_i^{(b_k)}=L$ for some $k=1,\ldots,m$, and let $H(F)_1, H(F)_0$ be the hyperplanes associated to $F$. 
Then $H(F)_1$ isolates $\Newt(G_{h,\lambda}(\x))$ if $|I|\leq k$, and $H(F)_0$ isolates $\Newt(G_{h,\lambda}(\x))$ if $|I|\geq k$.
\end{lemma}
\begin{proof}
As each $\Newt(s_{\lambda^{(b_\ell)}}(\x))$ is the convex hull of the $S_m$-orbit of $\lambda^{(b_\ell)}$, it suffices to show that the $S_m$-orbits of $\lambda^{(b_\ell)}$ for $\ell=1,\ldots, m$ all lie on one side of 
the proposed  $\Newt(G_{h,\lambda}(\x))$-isolating hyperplane.

Let $\pi\in S_m$, so that $\pi(\lambda^{(b_\ell)}) = (\lambda_{\pi^{-1}(1)}^{(b_\ell)}, \ldots, \lambda_{\pi^{-1}(m)}^{(b_\ell)})$. Then
$$\sum_{i\in I} \pi(\lambda^{(b_\ell)})_i
= \sum_{j \in \pi^{-1}(I)} \lambda^{(b_\ell)}_j
\leq \sum_{i=1}^{|I|} \lambda_i^{(b_\ell)},
$$
by Rado's inequalities.

As the partition $\lambda^{(b_k)}$ is obtained by adding the maximum allowable number ($a_\ell$) of boxes to the $\ell$-th row of $\lambda$ for $\ell=1,\ldots k$, then for $i\leq k$,
\begin{align*}
\lambda_i^{(b_\ell)} \leq \lambda_i^{(b_k)}, 
	&\hbox{ for } \ell \leq k,\\
\lambda_i^{(b_\ell)} = \lambda_i^{(b_k)}, 
	&\hbox{ for } \ell \geq k.
\end{align*}
And for $i>k$,
\begin{align*}
\lambda_i^{(b_\ell)} = \lambda_i^{(b_{k})}, 
	&\hbox{ for } \ell\leq k,\\
\lambda_i^{(b_\ell)} \geq \lambda_i^{(b_{k})}, 
	&\hbox{ for } \ell\geq k.	
\end{align*}
If $|I|\leq k$, then $H(F)_{1}$ isolates $\Newt(s_{\lambda^{(b_\ell)}}(\x))$ for all $\ell=0,\ldots,N$ since
$$\sum_{i\in I} \pi(\lambda^{(b_\ell)})_i
\leq \sum_{i=1}^{|I|} \lambda_i^{(b_\ell)}
\leq \sum_{i=1}^{|I|} \lambda_i^{(b_k)}
= L.$$
In other words, $\Newt(G_{h,\lambda}(\x)) \subseteq \overline{(H(F)_{1})_-}$.
Similarly, if $|I|\geq k,$ then $H(F)_{0}$ isolates $\Newt(s_{\lambda^{(b_\ell)}}(\x))$ for all $\ell=0,\ldots,N$ since
$$\sum_{i\notin I} \pi(\lambda^{(b_\ell)})_i
= \sum_{j\notin \pi^{-1}(I)} \lambda_j^{(b_\ell)}
\geq \sum_{i=|I|+1}^m \lambda_i^{(b_\ell)}
\geq \sum_{i=|I|+1}^m \lambda_i^{(b_k)}
= |\lambda^{(b_k)}|-L,$$
and $\Newt(G_{h,\lambda}(\x)) \subseteq \overline{(H(F)_{0})_+}$.
The result now follows.
\end{proof}

We will see that nearly all of the $\Newt(G_{h,\lambda}(\x))$-isolating hyperplanes of Lemma~\ref{lem.facetsofG} are facet-defining.  We first identify some that are not.
\begin{lemma}\label{lem.exclude1}
For $j=1,\ldots, m$, the affine hyperplanes
$$\mathcal{J}_j = \left\{ \x\in \mathbb{R}^m \mid \hbox{$\sum_{i\neq j} x_i = |\lambda|-\lambda_1$} \right\}$$
are not facet-defining hyperplanes of $\Newt(G_{h,\lambda}(\x))$.
\end{lemma}
\begin{proof}
By Lemma~\ref{lem.facetsofG}, the hyperplanes $\mathcal{J}_j$ isolate $\Newt(G_{h,\lambda}(\x))$.  
However, since $\sum_{i=2}^m \lambda_i^{(\ell)} $ $> \sum_{i=2}^m \lambda_i$ for all $\lambda^{(\ell)} \neq \lambda$,  then
$$\mathcal{J}_j \cap \Newt(G_{h,\lambda}(\x)) 
=\conv \left\{(\lambda_1, \lambda_{\sigma(2)}, \ldots, \lambda_{\sigma(m)})\mid \sigma \in S_{\{2,\ldots, m\}} \right\} \subseteq \Newt(s_\lambda(\x)),$$
so $\dim (\mathcal{J}_j \cap \Newt(G_{h,\lambda}(\x))) < m-1$ for $j=1,\ldots, m$, and $\mathcal{J}_j$ is not a facet-defining hyperplane of $\Newt(G_{h,\lambda}(\x))$.
\end{proof}

We can further narrow down the set of $\Newt(G_{h,\lambda}(\x))$-isolating hyperplanes of Lemma \ref{lem.facetsofG} that are facet-defining.
\begin{lemma}\label{lem.exclude2}
The facet-defining hyperplanes of $\Newt(s_{\lambda^{(b_m)}}(\x))$ that isolate $\Newt(G_{h,\lambda}(\x))$ are a subset of the facet-defining hyperplanes of $\Newt(s_{\lambda^{(b_{m-1})}}(\x))$ that isolate $\Newt(G_{h,\lambda}(\x))$.
\end{lemma}
\begin{proof}
Recall that $\lambda^{(b_m)}=\lambda^{(N)}$ and it is distinct from $\lambda^{(b_{m-1})}$ since we assumed that $\lambda_m=0$.  The only difference between the partitions $\lambda^{(b_{m-1})}$ and $\lambda^{(b_m)}$ is in their $m$-th row, so no additional $\Newt(G_{h,\lambda}(\x))$-isolating hyperplane that arises from a facet of $\Newt(s_{\lambda^{(b_m)}}(\x))$ is introduced.
\end{proof}

\begin{theorem}\label{thm.facetsofG}
Let $I$ be a proper nonempty subset of $[m]$ and let $F=F(I)$ be a facet of $\Newt(s_{\lambda^{(b_k)}}(\x))$ defined by 
$\sum_{i\in I} x_i \leq \sum_{i=1}^{|I|} \lambda_i^{(b_k)}=L$ for some $k=1,\ldots,m$.
Suppose $H(F)$ is a hyperplane associated to $F$ which isolates $\Newt(G_{h,\lambda}(\x))$ and is not of the form $\mathcal{J}_j$.  Then $H(F)$ is a facet-defining hyperplane of $\Newt(G_{h,\lambda}(\x))$. 
\end{theorem}
\begin{proof}
First, suppose $H(F)$ is of the form $H(F)_1$ so that by Lemma~\ref{lem.facetsofG}, $\Newt(G_{h,\lambda}(\x))$ lies in the half-space
$$\overline{(H(F)_1)_-} = \left\{\x\in \mathbb{R}^m \mid \hbox{$\sum_{i\in I} x_i \leq \sum_{i=1}^{|I|} \lambda_i^{(b_k)}=L$} \right\},$$
with $|I|\leq k$. Furthermore by Lemma~\ref{lem.exclude2}, it suffices to assume $k=1,\ldots, m-1$.
So if $\pi\in S_m$ is a permutation such that $\pi^{-1}(I)= \{1,\ldots, |I|\}$, then
$$\sum_{i\in I} \pi(\lambda^{(b_k+1)})_i 
= \sum_{j\in \pi^{-1}(I)} \lambda^{(b_k+1)}_j 
= \sum_{i=1}^{|I|} \lambda^{(b_k+1)}_i
= \sum_{i=1}^{|I|} \lambda_i^{(b_k)}
=L,$$
noting that the partition $\lambda^{(b_k+1)}$ exists since $k\leq m-1$.
Thus $\pi(\lambda^{(b_k+1)})$ is a point that lies on $H(F)_1 \cap \Newt(G_{h,\lambda}(\x))$, but not on $\Newt(s_{\lambda^{(b_k)}}(\x))$.

Second, suppose $H(F)$ is of the form $H(F)_0$ so that by Lemma~\ref{lem.facetsofG}, $\Newt(G_{h,\lambda}(\x))$ lies in the half-space
$$\overline{(H(F)_0)_+} = \left\{\x\in \mathbb{R}^m \mid \hbox{$\sum_{i\notin I} x_i \geq \sum_{i=|I|+1}^m \lambda_i^{(b_k)}=|\lambda^{(b_k)}|-L$} \right\},$$
with $|I|\geq k$. 
Furthermore by Lemma~\ref{lem.exclude2}, it suffices to assume $k=2,\ldots, m$.
So if $\pi\in S_m$ is a permutation such that $\pi^{-1}(I)= \{1,\ldots, |I|\}$, then
$$\sum_{i\notin I} \pi(\lambda^{(b_k-1)})_i 
= \sum_{j\notin \pi^{-1}(I)} \lambda^{(b_k-1)}_j 
= \sum_{i=|I|+1}^{m} \lambda^{(b_k-1)}_i
= \sum_{i=|I|+1}^{m} \lambda^{(b_k)}_i
=|\lambda^{(b_k)}|-L,$$
noting that the partition $\lambda^{(b_k-1)}$ exists since $k\geq 2$.
Thus $\pi(\lambda^{(b_k-1)})$ is a point that lies on $H(F)_0 \cap \Newt(G_{h,\lambda}(\x))$, but not on $\Newt(s_{\lambda^{(b_k)}}(\x))$.

In both of these cases, note that $\dim(H(F) \cap \Newt(s_{\lambda^{(b_k)}}(\x))) = m-2$ because it is a facet of $\Newt(s_{\lambda^{(b_k)}}(\x))$.  
Also, $\Newt(s_{\lambda^{(b_k)}}(\x))$ lies on the hyperplane $x_1+\cdots +x_m = |\lambda^{(b_{k})}|$ while $\p=\pi(\lambda^{(b_k\pm1)})$ is a point on $H(F)$ that lies on $x_1+\cdots +x_m = |\lambda^{(b_k\pm1)}|$.
Thus
$$m-1 \geq \dim(H(F)\cap \Newt(G_{h,\lambda}(\x))) 
\geq \dim \conv \left(\left( H(F)\cap \Newt(s_{\lambda^{(b_k)}}(\x)) \right) \cup \p \right) > m-2.$$
Therefore, $\dim(H(F) \cap \Newt(G_{h,\lambda}(\x))) = m-1$ and $H(F)$ is a facet-defining hyperplane of $\Newt(G_{h,\lambda}(\x))$.
\end{proof}

\begin{corollary}\label{cor.soneisreflexive}
If $\Newt(G_{h,\lambda}(\x))$ is reflexive, then so is $\Newt(s_{\lambda^{(1)}}(
\x))$. 
\end{corollary}
\begin{proof}
If $\Newt(G_{h,\lambda}(\x))$ is reflexive, then by Corollary~\ref{cor.uonone} its unique interior lattice point $\u$ lies on $\Newt(s_{\lambda^{(1)}}(\x))$, so $\u$ is also the unique interior lattice point of $\Newt(s_{\lambda^{(1)}}(\x))$.

Suppose $H$ is an affine hyperplane in $\mathbb{R}^m$ such that $H$ is a facet-defining hyperplane of $\Newt(s_{\lambda^{(1)}}(\x))$. 
The partitions $\lambda^{(0)}$ and $\lambda^{(1)}$ differ by exactly one box, and if this occurs in the $r$-th row, then $\lambda^{(0)} \subset \lambda^{(1)} \subseteq \lambda^{(b_r)}$ for some $r\geq2$. 
By the arguments in Theorem~\ref{thm.facetsofG}, if $\Newt(s_{\lambda^{(1)}}(\x))\subseteq \overline{H_-}$, then $H$ is a facet-defining hyperplane of $\Newt(s_{\lambda^{(b_r)}}(\x))$, and if $\Newt(s_{\lambda^{(1)}}(\x))\subseteq \overline{H_+}$, then $H$ is a facet-defining hyperplane of $\Newt(s_{\lambda^{(0)}}(\x))$.
Moreover, Theorem~\ref{thm.facetsofG} states that in either case, $H$ is a facet-defining hyperplane of $\Newt(G_{h,\lambda}(\x))$, and since it is reflexive, then $\u$ is lattice distance one from $H$.  Thus it follows that $\Newt(s_{\lambda^{(1)}}(\x))$ is reflexive.
\end{proof}

\begin{corollary}\label{cor.lambda_and_u}
If $\Newt(G_{h,\lambda}(\x))$ is reflexive and $\u=(u,\ldots,u)$ is its unique interior lattice point, then $\lambda_1 = u+1$.
\end{corollary}
\begin{proof}
By Theorem~\ref{thm.facetsofG}, $x_i = \lambda_1$ is a facet-defining hyperplane of $\Newt(G_{h,\lambda}(\x))$ for $i=1,\ldots,m$.  
If $\Newt(G_{h,\lambda}(\x))$ is reflexive, then $\u$ is lattice distance one from these hyperplanes, so in particular, $|\lambda_1-u|=1$.
Note that $\u$ and $\lambda^{(1)}$ form a primitive pair, so $u < \lambda_1^{(1)} = \lambda_1$ implies $\lambda_1 = u+1$.
\end{proof}

\begin{example}
Let $\lambda = (2,1,0)$ and $h=1$. Recall from Example~\ref{eg.210a} that
$$\textcolor{blue}{\lambda^{(b_1)} = \ydiagram{2,1}}\quad
\textcolor{green}{\lambda^{(b_2)} = \ydiagram{2,2}} \quad\hbox{and}\quad
\textcolor{orange}{\lambda^{(b_3)} = \ydiagram{2,2,2}}$$
The facet-defining inequalities from Theorem~\ref{prop:facets} for the Newton polytopes $\Newt(s_{\lambda^{(b_k)}}(\x))$ are (the $S_3$ permutations of) the following:
$$\begin{array}{l|ccc}
& H(F)_{1} && H(F)_{0} \\\hline
\lambda^{(b_1)}
	& x_1 \leq 2	& \hbox{ or } & x_2+x_3 \geq 1\\
	& x_1+x_2 \leq 3 & \hbox{ or } & x_3 \geq 0\\
	\hline
\lambda^{(b_2)}
	& x_1\leq 2 & \hbox{ or } & x_2+x_3 \geq 2\\
	\hline
\lambda^{(b_3)}
	& \emptyset && \emptyset	
\end{array}
$$
Lemma~\ref{lem.facetsofG} states that $x_i =2$, $x_i+x_j =1$, and $x_i=0$ are the hyperplanes which isolate $\Newt(G_{1,\lambda}(\x))$.
Theorem~\ref{thm.facetsofG} states that only $x_i \leq2$ and $x_i\geq0$ are facet-defining inequalities of $\Newt(G_{1,\lambda}(\x))$; the hyperplanes $x_i+x_j=1$ isolate $\Newt(G_{1,\lambda}(\x))$, but are not facet-defining.
\end{example}

It remains to prove that the facet-defining inequalities of $\Newt(G_{h,\lambda}(\x))$ are precisely the ones appearing in Lemma~\ref{lem.frontandback} and Theorem~\ref{thm.facetsofG}.

\begin{theorem}\label{thm.ineqsofG}
The facets of $\Newt(G_{h,\lambda}(\x))$ are determined by the following inequalities.
\begin{enumerate}
\item[(a)] $\displaystyle \sum_{i=1}^m x_i \geq |\lambda|$.
\item[(b)] $\displaystyle \sum_{i=1}^m x_i \leq |\lambda|+N$, if the $S_m$-orbit of $\lambda^{(N)}$ is nontrivial.
\item[(c)] 
$\displaystyle \sum_{i\in I}x_i \leq \sum_{i=1}^{|I|}\lambda_i^{(b_k)}$, if 
$I$ is a nonempty proper subset of $[m]$, $|I|\leq k$, 
and the inequality is a facet-defining inequality of 
$\Newt(s_{\lambda^{(b_k)}}(\x))$.
\item[(d)] 
$\displaystyle \sum_{i\notin I}x_i \geq \sum_{i=|I|+1}^m\lambda_i^{(b_k)}$, if
$I$ is a nonempty proper subset of $[m]$, 
$|I|\geq k$, and the inequality is a facet-defining inequality of 
$\Newt(s_{\lambda^{(b_k)}}(\x))$.
\end{enumerate}
\end{theorem}
\begin{proof}

Suppose $H$ is a facet-defining hyperplane of $\Newt(G_{h,\lambda}(\x))$. Then there exists a $k$ such that $H \cap \Newt(s_\lambda^{(b_k)})\neq \emptyset$ and $H$ isolates $\Newt(s_\lambda^{(b_k)})$. 
Thus $H \cap \Newt(s_\lambda^{(b_k)})$ is a face of $\Newt(s_\lambda^{(b_k)})$  and by Rado's inequalities, $H$ is defined by $\sum_{i\in I} x_i = a$ for some $a >0$ and $I\subset [m].$
Recall
\begin{align*} \overline{H_+} &= \left\{\x \in \mathbb{R}^m \mid \hbox{$\sum_{i\in I} x_i\geq a$}\right\},\\
\overline{H_-} &= \left\{\x \in \mathbb{R}^m \mid \hbox{$\sum_{i\in I} x_i\leq a$}\right\}.
\end{align*}

Suppose $H$ is a facet-defining hyperplane of $\Newt(G_{h,\lambda}(\x))$ so that, in particular, 
$$m-1 
= \dim (H\cap \Newt(G_{h,\lambda}(\x)))
\geq \dim (H\cap \Newt(s_{\lambda^{(b_k)}}(\x)))$$
for $k=1,\ldots, m$.

Suppose $\dim (H\cap \Newt(s_{\lambda^{(b_k)}}(\x)))=m-1$ for some $b_k$.
If $\Newt(G_{h,\lambda}(\x)) \subseteq \overline{H_+}$, then $b_1=0$ and $H\cap \Newt(s_{\lambda^{(0)}}(\x)) = \Newt(s_{\lambda^{(0)}}(\x))$, since $\dim \Newt(s_{\lambda^{(0)}}(\x)) = m-1$.  
This case corresponds to the facet-defining inequality $\sum_{i=1}^m x_i \geq |\lambda|$.

Otherwise, if $\Newt(G_{h,\lambda}(\x)) \subseteq \overline{H_-}$, then $b_m=N$ and $$H\cap \Newt(s_{\lambda^{(N)}}(\x)) = \Newt(s_{\lambda^{(N)}}(\x)).$$
If the $S_m$-orbit of $\lambda^{(N)}$ is nontrivial, then $\dim \Newt(s_{\lambda^{(N)}}(\x)) = m-1$.
This case corresponds to the facet-defining inequality $\sum_{i=1}^m x_i \leq |\lambda|+N$.

Now suppose $\dim (H\cap \Newt(s_{\lambda^{(b_\ell)}}(\x)))<m-1$ for all $\ell=1,\ldots,m$ . 
It remains to show that there exists $1\leq k \leq m$ such that $\dim (H\cap \Newt(s_{\lambda^{(b_k)}}(\x)))=m-2$.

Since the symmetric group $S_m$ acts on $\Newt(G_{h,\lambda}(\x))$ (and hence its facet-defining hyperplanes) then without loss of generality, it suffices to assume that $H$ is defined by an equation of the form $\sum_{i=1}^r x_i = d$ for some $r \leq m$.

First assume that $\Newt(G_{h,\lambda}(\x)) \subseteq \overline{H_-}$.
Since $H$ is facet-defining, then it contains at least one vertex of $\Newt(G_{h,\lambda}(\x))$, say $\lambda^{(b_k)}$, where we choose $k$ to be the smallest index for which this is true.
By the same argument as in the proof of Lemma~\ref{lem.exclude2}, we may assume that $k\leq m-1$.
This implies $d = \sum_{i=1}^r \lambda_i^{(b_k)}$. Since $\lambda_i^{(b_k)} \leq \lambda_i^{(b_\ell)}$ for all $i$ and all $\ell\geq k$, then
$$d = \sum_{i=1}^r \lambda_i^{(b_k)} \leq  \sum_{i=1}^r \lambda_i^{(b_\ell)} \leq d,$$
where the second inequality is due to $\Newt(G_{h,\lambda}(\x)) \subseteq \overline{H_-}$.
So the vertices $\lambda^{(b_\ell)}$ also lie in $H$ for all $\ell\geq k$.

From this observation, we see that the vertices of $\Newt(G_{h,\lambda}(\x))$ that lie in $H$ consists of concatenating an $S_r$-permutation of $(\lambda_1^{(b_k)},\ldots, \lambda_r^{(b_k)})$ with an $S_{m-r}$-permutation of $(\lambda_{r+1}^{(b_\ell)},\ldots, \lambda_m^{(b_\ell)})$ for any $\ell\geq k$.
The convex hull of the permutations of $(\lambda_{r+1}^{(b_\ell)},\ldots, \lambda_m^{(b_\ell)})$ for $\ell=k,\ldots, m$ has dimension $m-r$ as $k<m$ and $\lambda_m^{(b_k)}=0\neq \lambda_m^{(b_m)}$.
Let $s$ be the dimension of the convex hull of the permutations of $(\lambda_1^{(b_k)},\ldots, \lambda_r^{(b_k)})$, so that
$$\dim ( H \cap \Newt(G_{h,\lambda}(\x)) ) = s+m-r.$$
But $H$ is a facet-defining hyperplane of $\Newt(G_{h,\lambda}(\x))$, so $s+m-r=m-1$ implies $s=r-1$.

Now, the dimension of $H \cap \Newt(s_{\lambda^{(b_k)}}(\x))$ is the sum of $s$ and the dimension of the convex hull of the permutations of $(\lambda_{r+1}^{(b_k)},\ldots, \lambda_m^{(b_k)})$, which is $m-r-1$ as $k<m$.  
Thus
$$\dim ( H \cap \Newt(s_{\lambda^{(b_k)}}(\x)) ) = (r-1) + (m-r-1) = m-2.$$
Thus $H$ is a facet-defining hyperplane of $\Newt(s_{\lambda^{(b_k)}}(\x))$.

A similar argument works in the case $\Newt(G_{h,\lambda}(\x)) \subseteq \overline{H_+}$. Therefore the result follows.

\end{proof}

We are now ready to classify the Newton polytopes of inflated symmetric Grothendieck polynomials that are reflexive.  First, we examine the simple case when $m=2$.

\begin{proposition}\label{prop.m2}
Let $m=2$, and assume that $\lambda=(\lambda_1,0)$.  The only Newton polytopes of inflated symmetric Grothendieck polynomials that are reflexive are $\Newt(G_{h,(3,0)}(\x))$ for $h\geq2$.
\end{proposition}
\begin{proof}
If $\Newt(G_{h,\lambda}(\x))$ is reflexive, then by Corollary~\ref{cor.uonone} its unique interior lattice point $\u = (u,u)$ lies in $\Newt(s_{\lambda^{(1)}}(\x))$ so $2u = |\lambda|+1$.
Since $\lambda$ is assumed to be reduced by translation, then $\lambda = (2u-1,0)$.  By Corollary~\ref{cor.lambda_and_u} we know $\lambda_1=u+1$, so
$2u-1=u+1$ implies $u=2$ and $\lambda=(3,0)$.
When $h=1$, $\lambda^{(N)}=\lambda^{(1)}= (3,1)$ and $\Newt(G_{1,\lambda}(\x))$ has no interior lattice points. 

When $h=2$, $\Newt(G_{2,\lambda}(\x)) = \conv \{(3,0), (3,2), (2,3), (0,3) \}$ with the unique interior lattice point $(2,2)$ so it is reflexive.

Finally when $h\geq3$, $\Newt(G_{2,\lambda}(\x)) = \conv \{(3,0), (3,3), (0,3) \}$ with the unique interior lattice point $(2,2)$ so it is also reflexive.
\end{proof}

\begin{proposition} \label{prop.classifyrefG}
Let $m\geq3$ and assume that $\lambda$ is reduced by translation so that $\lambda = (\lambda_1,\ldots, \lambda_{m-1},0)$. 
If $\mathrm{Newt}(G_{h,\lambda}(\x))$ is reflexive, then $\lambda$ is of one of the following forms:
$$\lambda = \begin{cases}
(m+1,\ldots, m+1,0) \vdash m^2-1,\\
(2,\ldots,2,1,0,\ldots,0) \vdash m-1, &\hbox{ if $m$ is even,}\\
(2,\ldots, 2,0,\ldots,0) \vdash m-1, &\hbox{ if $m$ is odd.}
\end{cases}$$
\end{proposition}
\begin{proof}
By Corollary~\ref{cor.soneisreflexive}, if $\mathrm{Newt}(G_{h,\lambda}(\x))$ is reflexive then $\Newt(s_{\lambda^{(1)}}(\x))$ is reflexive, so we know from Theorem~\ref{thm:reflexive} that the partition ${\lambda^{(1)}}$ is a translate of one of
$$(m,0,\ldots,0), (2,1^{m-2},0), (2^\ell,0^\ell), (2^\ell,1,0^\ell)\vdash m, 
\hbox{ or } (m,\ldots, m,0)\vdash m(m-1).$$

First suppose $\lambda^{(1)}$ is a translate of $(m,0,\ldots,0)$ so that it is of the form $(m+c,c,\ldots,c)$ for some nonnegative integer $c$. 
Then $\lambda = (m+c,c,\ldots,c,c-1)$ or $(m+c-1,c\ldots,c)$. 
Since $\lambda$ is assumed to be reduced by translation, then $c=1$ and $\lambda = (m+1,1,\ldots,1,0)$ or $c=0$ and $\lambda = (m-1,0,\ldots,0)$.  
In the first case, this implies that the box that was added to $\lambda$ to obtain $\lambda^{(1)}$ is in the $m$-th row (where $m\geq3$), but by definition it should have been in the second row so this case is not possible.  
In the second case, this implies that the box that was added to $\lambda$ to obtain $\lambda^{(1)}$ is in the first row but this is also not possible by definition.

Next, suppose $m\geq4$ and $\lambda^{(1)}$ is a translate of $(2,1^{m-2},0)$ so that it is of the form $(2+c,(1+c)^{m-2},c)$ for some nonnegative integer $c$.
Then $\lambda = (2+c, (1+c)^{m-2}, c-1)$, $(2+c, (1+c)^{m-3},c,c)$, or $((1+c)^{m-1},c)$. 
Since $\lambda$ is assumed to be reduced by translation, then $c=1$ in the first case and $c=0$ in the second and third cases.   
All three of these cases are not possible for reasons analogous to the previous case.
When $m=3$, the only possible case is when $\lambda^{(1)} = (2,1,0)$ and $\lambda = (2,0,0)$.  This case is also covered in one of the cases below.

Now, suppose $\lambda^{(1)}$ is a translate of $(2^\ell,0^\ell)\vdash m=2\ell$ or $(2^\ell,1,0^\ell)\vdash m = 2\ell+1$ so that it is of the form $((2+c)^\ell, c^\ell)$ or $((2+c)^\ell,1+c,c^\ell)$.  
When $m$ is even, then $\lambda = ((2+c)^\ell, c^{\ell-1}, c-1)$ or $((2+c)^{\ell-1}, 1+c, c^\ell)$. 
Since $\lambda$ is assumed to be reduced by translation, then $c=1$ and $\lambda = (3^\ell,1^{\ell-1},0)$ in the first case, or $c=0$ and $\lambda=(2^{\ell-1},1,0^\ell)\vdash m-1$ in the second.
In the first case, this implies that the box that was added to $\lambda$ to obtain $\lambda^{(1)}$ is in the $m$-th row (where $m\geq4$), when it should have been added in the $(\frac{m}{2}+1)$-th row, so this case may be eliminated.
When $m$ is odd, then $\lambda=((2+c)^\ell,1+c,c^{\ell-1},c-1)$, $((2+c)^\ell,c^{\ell+1})$ or $((2+c)^{\ell-1},(1+c)^2,c^\ell)$.
Since $\lambda$ is assumed to be reduced by translation, then $c=1$ in the first case and $c=0$ in the second and third cases.  By arguments similar to those in the previous cases, the only possible case is the second, when $\lambda=(2^\ell,0^{\ell+1})\vdash m-1$.

Finally, suppose $\lambda^{(1)}$ is a translate of $(m,\ldots,m,0)$, so that it is of the form $(m+c,\ldots,m+c,c)$ for some nonnegative integer $c$.
Then $\lambda = (m+c,\ldots,m+c,c-1)$ or $(m+c,\ldots,{m+c-1},c)$.  Since $\lambda$ is assumed to be reduced by translation, then $c=1$ and $\lambda=(m+1,
\ldots,m+1,0)$ in the first case, or $c=0$ and $\lambda= (m,\ldots,m, m-1,0)$ in the second.
To rule out this latter possibility, note that if $\Newt(G_{h,\lambda}(\x))$ is reflexive, then by Corollary~\ref{cor.uonone} its unique interior lattice point is $\u=(u,\ldots,u)\vdash mu=|\lambda|+1= m^2-m$, so $u=m-1$.  By Theorem~\ref{thm.ineqsofG}, $x_i\geq0$ is a facet-defining inequality of $\Newt(G_{h,\lambda}(\x))$, and $\u$ is lattice distance $m-1\geq2$ from each of these hyperplanes, so $\Newt(G_{h,\lambda}(\x))$ cannot be reflexive.

The result now follows.
\end{proof}

We next determine the values of $h$ for which the partitions $\lambda$ in Proposition~\ref{prop.classifyrefG} give rise to reflexive Newton polytopes of inflated symmetric Grothendieck polynomials.

\begin{proposition}\label{prop.reflG1}
Let $\lambda=(m+1,\ldots,m+1,0)\vdash m^2-1$.  Then the Newton polytope $\Newt(G_{h,\lambda}(\x))$ is reflexive when $h=1$ and $m=3$, or $h\geq2$ and $m\geq3$.
\end{proposition}
\begin{proof} 
We first deduce where an interior lattice point can occur in $\Newt(G_{h,\lambda}(\x))$.
The sequence of dominating partitions for $G_{h,\lambda}(\x)$ is
$$\lambda^{(k)} = (m+1,\ldots, m+1,k)$$ for $k=0,\ldots, N=\min(h(m-1),m+1)$. 
If the partition $\mu=(\mu_1,\ldots, \mu_m)$ is an interior lattice point of $\Newt(G_{h,\lambda}(\x))$, then it lies in $\Newt(s_{\lambda^{(r)}}(\x))$ for some $r=1,\ldots, N-1$, and so $\mu \vdash |\lambda|+r = m^2-1+r$.
Moreover, $\mu$ is an interior lattice point only if 
$$\sum_{i=1}^\ell \mu_i < \sum_{i=1}^\ell \lambda_i^{(k)} = \ell(m+1)$$
for every $\ell=1,\ldots, m-1$ and $k=1,\ldots, N-1$.  
In particular when $\ell=1$, this implies $m\geq\mu_1\geq\cdots\geq\mu_m$, so it follows that $\sum_{i=1}^\ell \mu_i \leq \ell m$ for $\ell=1,\ldots, m$.
Therefore, 
$$m^2-1+r = |\lambda|+r = \sum_{i=1}^m \mu_i \leq m^2, $$
and we conclude that $r=1$ necessarily.
This means $\mu \vdash m^2$ so the only possible $\mu$ is $\mu = (m,\ldots, m)$, whose $S_m$-orbit is trivial.
Therefore when $\lambda = (m+1,\ldots, m+1,0)$, the polytope $\Newt(G_{h,\lambda}(\x))$ has the unique interior lattice point $\u=(m,\ldots, m)$.


If the $S_m$-orbit of $\lambda^{(N)}$ is not trivial, then by Lemma~\ref{lem.frontandback}, $x_1+\cdots+x_m=|\lambda|+N=m^2-1+N$ is a facet-defining hyperplane of $\Newt(G_{h,\lambda}(\x))$. 
Thus for $\Newt(G_{h,\lambda}(\x))$ to be reflexive, we require $N=2$, as the interior lattice point $\u=(m,\ldots,m)$ must be lattice distance one from this hyperplane.

In other words, $N=\min(h(m-1),m+1)=2$.  
Since we assumed $m\geq3$, then $h=1$ and $m=3$ is the only possibility. 
In this case, $\lambda = (4,4,0)$ and by Theorem~\ref{thm.ineqsofG} the facet-defining hyperplanes of this Newton polytope are 
$$x_i=4 \hbox{ for } i=1,2,3, \quad x_1+x_2+x_3=8, \quad \hbox{ and }\quad x_1+x_2+x_3=10.$$  
The unique interior lattice point $\u=(3,3,3)$ is lattice distance one from each of these facet-defining hyperplanes, so $\Newt(G_{1,(4,4,0)}(\x))$ is reflexive. 

Otherwise, if the $S_m$-orbit of $\lambda^{(N)}$ is trivial, then we must have $\lambda^{(N)}=(m+1,\ldots, m+1)$.  This implies $h(m-1)\geq m+1$, or equivalently, $h\geq 1+\frac{2}{m-1}$.  As $m\geq3$, then $h\geq2$.

In this case, we have $\lambda^{(b_1)} = \cdots = \lambda^{(b_{m-1})} = ({m+1},\ldots, m+1,0)$ and $\lambda^{(b_m)}=(m+1,\ldots, m+1)$. 
By Theorem~\ref{thm.ineqsofG}, the facet-defining hyperplanes of $\Newt(G_{h,\lambda}(\x))$ are
$$x_1+\cdots+x_m=m^2-1,\quad \hbox{ and }\quad x_i = m+1  \hbox{ for } i=1,\ldots, m,$$
and we see that this Newton polytope is an $m$-simplex.
The unique interior lattice point $\u = (m,\ldots, m)$ is lattice distance one from each of these hyperplanes so $\Newt(G_{h,\lambda}(\x))$ is reflexive.
\end{proof}

\begin{proposition}\label{prop.reflG2}
Let $m\geq3$ and let
$$\lambda = \begin{cases}
(2^\ell, 0^{\ell+1})\vdash m-1, &\hbox{ if $m=2\ell+1$ is odd,}\\
(2^{\ell-1},1,0^\ell)\vdash m-1, &\hbox{ if $m=2\ell$ is even.}
\end{cases}$$
Then $\Newt(G_{h,\lambda}(\x))$ is reflexive for $h\geq1$.
\end{proposition}
\begin{proof}
We first deduce where an interior lattice point can occur in $\Newt(G_{h,\lambda}(\x))$. 
For all $h\geq1$ and $m\geq3$, if $\lambda^{(0)}, \ldots, \lambda^{(N)}$ is the sequence of dominating partitions for $G_{h,\lambda}(\x)$, then $N=m+1$. 
In particular, $\lambda_1^{(k)}=2$ for all $k=0,\ldots, N$.  
If the partition $\mu=(\mu_1,\ldots,\mu_m)$ is an interior lattice point of $\Newt(G_{h,\lambda}(\x))$ that lies in $\Newt(s_{\lambda^{(r)}}(\x))$ for some $r=1,\ldots, N-1$, then $\mu \vdash |\lambda|+r = m-1+r$.  As $\mu$ is an interior lattice point, it must satisfy $\mu_1 < \lambda_1^{(k)}=2$, so we have $1\geq \mu_1 \geq \cdots \geq \mu_m$ and
$$m-1+r = |\lambda|+r = \mu_1+\cdots+\mu_m \leq m.$$
This implies $r=1$, and the only possible $\mu$ is $\mu=(1,\ldots,1)
\vdash m$, whose $S_m$-orbit is trivial. The point $\mu=(1,\ldots,1)$ satisfies $\sum_{i=1}^\ell\mu_i < \sum_{i=1}^\ell \lambda_i^{(k)}$ for all $k=1,\ldots, N-1$; thus $\Newt(G_{h,\lambda}(\x))$ has the unique interior lattice point $\u=(1,\ldots,1)$.

First consider the case $m=3$, $h=1$, and $\lambda=(2,0,0)$. 
The dominating sequence of partitions is $(2,0,0)$, $(2,1,0)$ and $(2,1,1)$.  
The facet-defining hyperplanes of $\Newt(G_{h,\lambda}(x))$  are 
$$x_i=0,2 \hbox{ for } i=1,2,3,\quad
 x_i+x_j = 3 \hbox{ for } i\neq j \in [3],\quad
 \hbox{ and }\quad x_1+x_2+x_3=2,4.$$

Evidently, $\u$ is lattice distance one from every facet-defining hyperplane of the Newton polytope, so $\Newt(G_{h,\lambda}(\x))$ is reflexive.  

Next consider the cases $m=3$ and $h\geq2$, or $m\geq4$ and $h\geq1$. 
In all cases, we have $N= m+1$, with $\lambda^{(N)}= (2^m)$, and the Newton polytope is the same for all $h$ in the given range. 
The vertices of $\Newt(G_{h,\lambda}(\x))$ are the $S_m$-orbits of $\lambda^{(0)}, \lambda^{(2)}, \ldots, \lambda^{(N)}$ if $m$ is odd, and are the $S_m$-orbits of $\lambda^{(0)}, \lambda^{(1)}, \lambda^{(3)}, \ldots, \lambda^{(N)}$ if $m$ is even.

By Theorem~\ref{thm.ineqsofG}, the facet-defining hyperplanes of $\Newt(G_{h,\lambda}(\x))$ are
$$x_i=0,2 \hbox{ for } i=1,\ldots,m,\quad \hbox{ and }\quad x_1+\cdots+x_m=m-1,$$
and we see that $\Newt(G_{h,\lambda}(\x))$ is the truncation of the $m$-cube $[0,2]^m$ by the hyperplane $x_1+\cdots+x_m = m-1$.
The unique interior lattice point $\u=(1,\ldots,1)$ is lattice distance one from each of these hyperplanes, and we conclude that $\Newt(G_{h,\lambda}(\x))$ is reflexive. 
\end{proof}

\begin{theorem}
The Newton polytope $\Newt(G_{h,\lambda}(\x))$ is reflexive if and only if $h$  and $\lambda=(\lambda_1,\ldots, \lambda_m)$ are one of the following cases:
$$\begin{array}{lll}
h\geq1, 
	& \lambda=(2,\ldots,2,0,\ldots,0)\vdash m-1
	& \hbox{ for odd $m\geq3$},\\
h\geq1, 
	& \lambda=(2,\ldots,2,1,0,\ldots,0)\vdash m-1
	& \hbox{ for even $m\geq4$},\\
h\geq2, 
	& \lambda=(m+1,\ldots,m+1,0)\vdash m^2-1
	& \hbox{ for } m\geq2,\\
h=1, 
	& \lambda=(4,4,0).
\end{array}
$$
\end{theorem}
\begin{proof}
The result follows from combining Propositions~\ref{prop.m2}, \ref{prop.reflG1} and \ref{prop.reflG2}.
\end{proof}

Tables of $h^\ast$-vectors for reflexive polytopes in the family $\Newt(G_{h,\lambda}(\x))$ can be found in the Appendix.


\section{$h^\ast$-polynomials}

Thus far we have shown that the Newton polytopes arising from Schur polynomials have the integer decomposition property (Proposition~\ref{thm:Schur_IDP}). In addition, we have classified the polytopes that are reflexive (Theorem~\ref{thm:reflexive}) and Gorenstein (Corollary~\ref{cor:Schur_gorenstein}). Following the motivation of Conjecture~\ref{conj-unimodal}, we study the $h^\ast$-polynomials of these polytopes.

In this section, we provide closed-form expressions for the $h^\ast$-polynomials of the four families of reflexive Newton polytopes of Schur polynomials from Theorem~\ref{thm:reflexive}. 
We also prove that all of these $h^\ast$-polynomials are unimodal. Figure \ref{fig.Reflexive-Nonreflexive-h*} depicts examples of reflexive and nonreflexive $\Newt(s_\lambda(\x))$ with their corresponding $h^{\ast}$-vectors. The $h^\ast$-vectors of reflexive $\Newt(s_\lambda(\x))$ for several partitions are provided in Table~\ref{tab:h*-vecs} in the Appendix.

\begin{figure}[ht!]
\begin{center}
\includegraphics[height=5.8cm]{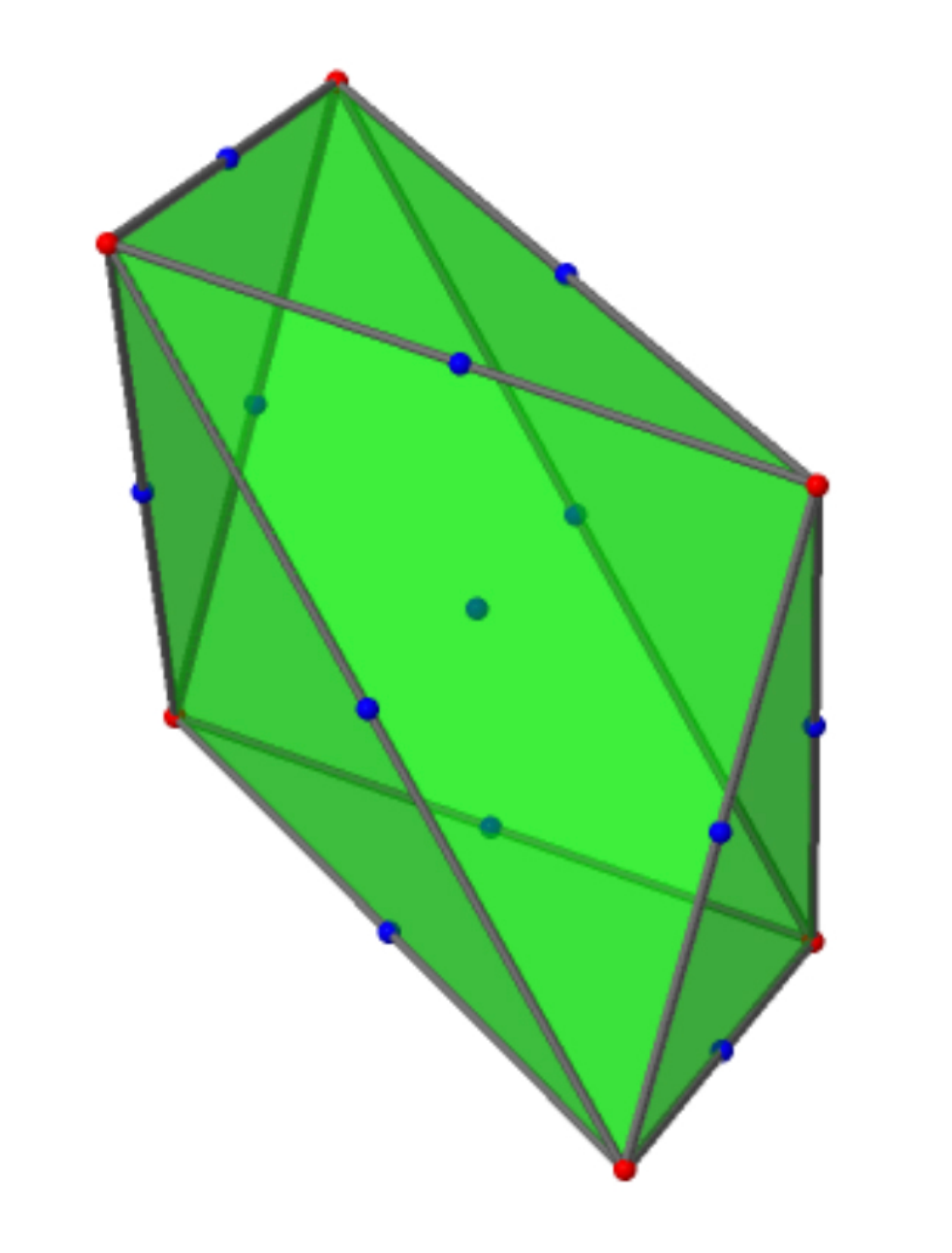} 
\hspace{1cm}
\includegraphics[height=6cm]{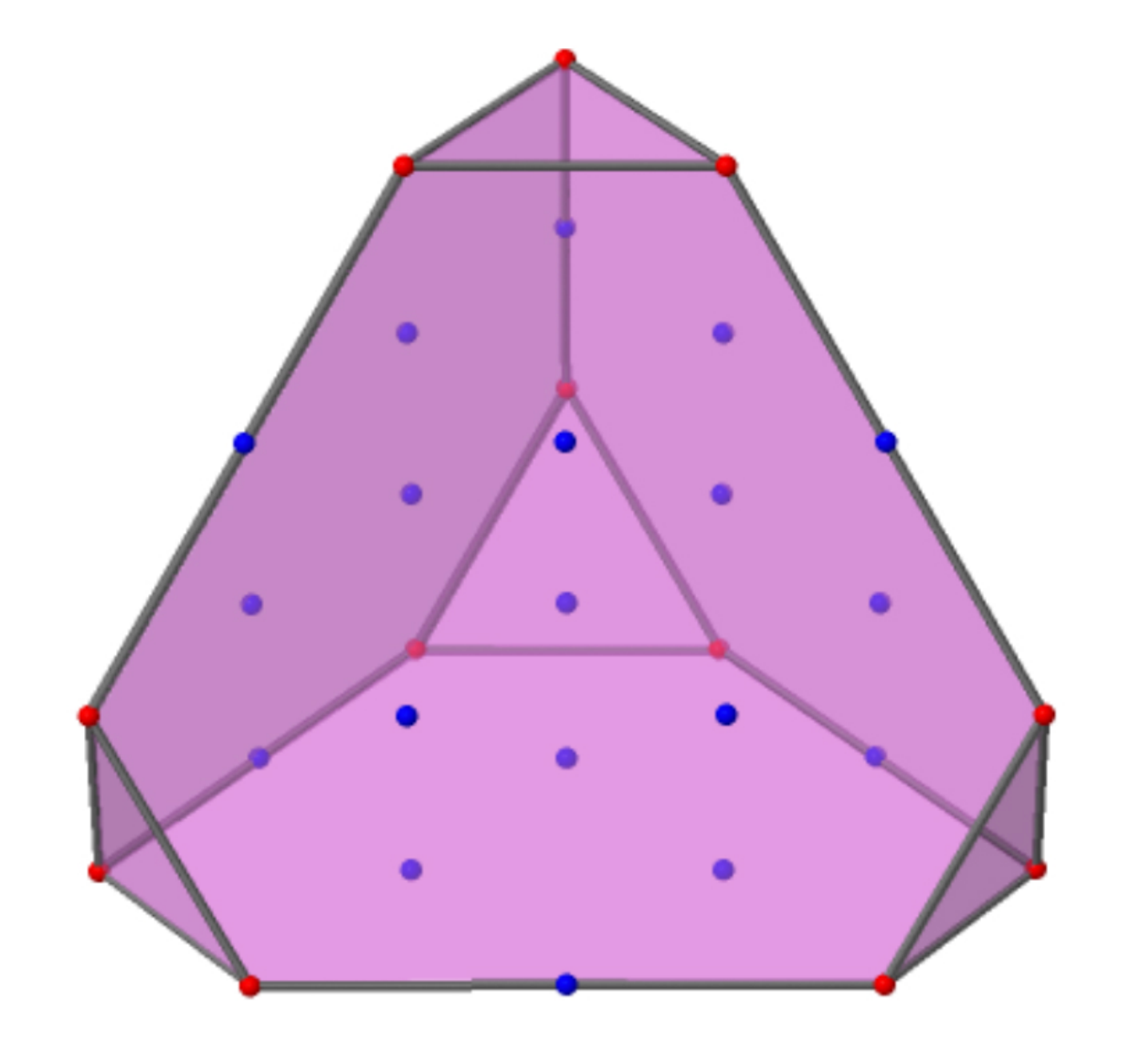} 
\caption{The reflexive polytope $\Newt(s_{(2,2,0,0)}(\x))$ has $h^{\ast}$-vector $(1,15,15,1)$ and is shown on the left. The polytope $\Newt(s_{(3,1,0,0)}(\x))$, which has $h^{\ast}$-vector $(1,27,31,1)$, is shown on the right and is not reflexive.}
\label{fig.Reflexive-Nonreflexive-h*}
\end{center}
\end{figure}

\begin{proposition}\label{prop:(n)_h*}
Let $\lambda= (n)$.
The $h^\ast$-polynomial of the Newton polytope $\Newt(s_\lambda(\x))=\mathcal{P}^n_\lambda$ is of degree $n-1$, and has coefficients
$$ h^\ast_j = \sum_{i=0}^j (-1)^i {n\choose i}{(j-i+1)n-1\choose n-1}.$$
\end{proposition}

\begin{proof}
Let $\mathcal{P} = \mathcal{P}^n_\lambda$.
We first calculate the Ehrhart polynomial $\ehr_\mathcal{P}(k)$.
The polytope $\mathcal{P}$ has vertices $ne_i$ for $1\le i\le n$ and the polytope
$k\mathcal{P}$ has vertices $kne_i$ for $1\le i\le n$, where $e_i$ is the $i^{th}$ standard basis vector.
The lattice points in $k\mathcal{P}$ are all the points 
$(v_1,v_2,\ldots, v_n)$ with all $v_i$ nonnegative integers and 
$\sum_i v_i = kn$.  These are in natural bijection with all 
weak compositions of $kn$ into $n$ parts.
The number of these is well known, so we get
$$\ehr_\mathcal{P}(k) = {kn+n-1 \choose n-1}.$$
Thus the Ehrhart series  of $\mathcal{P}$ is 
$$\sum_{k=0}^\infty \ehr_\mathcal{P}(k) x^k = \sum_{k=0}^\infty {kn+n-1 \choose n-1}x^k.$$
To get the $h^\ast$-polynomial, we multiply the Ehrhart series by $(1-x)^n$. 

\begin{eqnarray*}
(1-x)^n \sum_{k=0}^\infty \ehr_\mathcal{P}(k) x^k &=& 
\left(\sum_{i=0}^n (-1)^i{n\choose i}x^i\right)
\left(\sum_{k=0}^\infty {kn+n-1 \choose n-1}x^k\right)\\
&=&\sum_{j=0}^{n-1}
\left(\sum_{i=0}^j (-1)^i {n\choose i}{(j-i)n+n-1\choose n-1}\right)x^j.
\end{eqnarray*}

So for $0\le j\le n-1$, 
$$h^\ast_j = \sum_{i=0}^j (-1)^i {n\choose i}{(j-i+1)n-1\choose n-1}.$$

\end{proof}

\begin{remark}
The coefficients in Proposition~\ref{prop:(n)_h*} are precisely those given by OEIS Sequence A108267 \cite{oeis}.
\end{remark}

\begin{proposition}\label{prop:(2,1,...,1)_h*}
Let $\lambda =  (2,1,\ldots,1)\vdash n$ and $\mathcal{P}^n_\lambda$ be the corresponding Newton polytope. Its $h^\ast$ coefficients are given by the following formula.
\begin{eqnarray*}
h^\ast_j &=& {n-1\choose j}^2 
\end{eqnarray*}
\end{proposition}

\begin{proof}
Let $Q$ be the translation of $\mathcal{P}^n_\lambda$ so that the origin is its unique interior lattice point. Then $Q$ is the convex hull of all vectors $e_i - e_j$ for all $i,j \in [n]$. Thus $Q$ is precisely the polytope $\mathcal{P}_{A_{n-1}}$ as stated in \cite[Theorem 2]{Ardila-etal}. This theorem shows that $h^\ast_j = {n-1\choose j}^2.$ 
\end{proof}

\bigskip

\begin{proposition}\label{prop:(2,...,2)_h*}
Let $\lambda = (2,\ldots, 2)\vdash n$ ($n$ even) or $\lambda = (2,\ldots, 2,1)\vdash n$ ($n$ odd) and $\mathcal{P}^n_\lambda$ be the corresponding Newton polytope. The $h^\ast$ coefficients are given by the following formula. 
\begin{eqnarray*}
h^\ast_j &=& \sum_{k=0}^j (-1)^{j-k} {n\choose j-k} a_{n,k}, \text{ where } \\
a_{n,k} &=& \sum_{i=0}^n (-1)^i {n\choose i}
            {nk+n-1-i(2k+1) \choose n-1}. 
\end{eqnarray*}
\end{proposition}

\begin{proof}
We proceed in the same way as Proposition~\ref{prop:(n)_h*}. Let $n$ be even and $\lambda = (2,\ldots, 2)\vdash n$. Define $\mathcal{P} = \mathcal{P}^n_\lambda$. We first calculate the Ehrhart polynomial $\ehr_\mathcal{P}(k)$.
The vertices of $\mathcal{P}$ are $2\sum_{i \in I} e_i$ for all $I \subseteq [n]$ such that $|I| = n/2.$ Thus $k\mathcal{P}$ has vertices $2k\sum_{i \in I} e_i$ for these same $I$. Therefore the lattice points in $k\mathcal{P}$ are the points $(v_1,\dots,v_n)$ such that $0 \le v_i \le 2k$ and $\sum_i v_i = kn$. There are in natural bijection with the weak compositions of $kn$ into $n$ parts such that each part has size less than or equal to $2k$. This number is counted by $a_{n,k}$ above. Therefore $\ehr_\mathcal{P}(k) = a_{n,k}$.

Thus the Ehrhart series  of $\mathcal{P}$ is 
$$\sum_{k=0}^\infty \ehr_\mathcal{P}(k)x^k = \sum_{k=0}^\infty a_{n,k} x^k.$$
To get the $h^\ast$-polynomial, we multiply the Ehrhart series by $(1-x)^n$:
\begin{eqnarray*}
(1-x)^n \sum_{k=0}^\infty \ehr_\mathcal{P}(k)x^k &=& 
\left(\sum_{i=0}^n (-1)^i{n\choose i}x^i\right)
\left(\sum_{k=0}^\infty a_{n,k} x^k\right)
\end{eqnarray*}

So for $0\le j\le n-1$, the coefficient of $x^j$ in the above series is
$$
h^\ast_j = \sum_{k=0}^j (-1)^{j-k} {n \choose j-k} a_{n,k}.
$$

Now suppose that $n$ is odd and $\lambda = (2,\ldots, 2,1)\vdash n$. Again define $\mathcal{P} = \mathcal{P}^n_\lambda$. The vertices of $\mathcal{P}$ are $e_j + 2\sum_{i \in I}e_i$ for all $j \notin I \subseteq [n]$ with $|I|= \lfloor n/2 \rfloor.$ Thus $k\mathcal{P}$ has vertices $k(e_j + 2\sum_{i \in I}e_i)$ for these same $I$ and $j$. Therefore the lattice points in $k\mathcal{P}$ are the points $(v_1,\dots,v_n)$ such that $0 \le v_i \le 2k$ and $\sum_i v_i = kn$. Thus the $h^\ast$-polynomial calculation in the even case works precisely the same way in the odd case.
\end{proof}

Now that we have determined the $h^\ast$-polynomials for these reflexive polytopes, we verify that Conjecture~\ref{conj-unimodal} holds in the case of Newton polytopes of Schur polynomials. That is, we show that the coefficients of the $h^\ast$-polynomials of the polytopes listed in Theorem~\ref{thm:reflexive} are unimodal.

\begin{proposition}\label{prop:HibiOhsugi_Schur}
All Newton polytopes arising from Schur polynomials that are reflexive 
have $h^\ast$-polynomials with unimodal coefficients.
\end{proposition}

\begin{proof}

We note that if $\mathcal{P}$ is reflexive and $\mathcal{P}$ has a regular unimodular triangulation, then the $h^\ast$-polynomial of $\mathcal{P}$ has unimodal coefficients by Theorem~\ref{thm-GorensteinTriangulation}. Thus, we can prove this result by showing that these polytopes have regular unimodular triangulations. 

If $\mathcal{P}$ is the polytope from Proposition~\ref{prop:(n)_h*}, then $\mathcal{P}$ is a dilate of a unimodular simplex. By \cite[Theorem 4.8]{Haase_etal}, the dilate of any polytope with a regular unimodular triangulation has a regular unimodular triangulation.

If $\mathcal{P}$ is the polytope from Proposition~\ref{prop:(2,1,...,1)_h*}, we observe that the coefficients are log-concave by a routine calculation since binomial coefficients are log-concave.

If $\mathcal{P}$ is the polytope from Proposition~\ref{prop:(2,...,2)_h*} when $n$ is even, then $\mathcal{P}$ is a dilate of a hypersimplex, which is known to have a unimodular triangulation by \cite{LamPostnikov}. 

If $\mathcal{P}$ is the polytope from Proposition~\ref{prop:(2,...,2)_h*} when $n$ is odd, we note that the facets for this polytope are listed in Proposition~\ref{refl-list}(c). We can  translate $\mathcal{P}$ so that the origin is the unique interior point of $\mathcal{P}$. The facet-defining matrix for this translate of $\mathcal{P}$ is easily seen to be unimodular, so by \cite[Theorem 2.4]{Haase_etal}, $\mathcal{P}$ has a regular unimodular triangulation. Thus the conjecture holds for this family of polytopes.
\end{proof}

\begin{remark}
One could also consider the question of $h^\ast$-unimodality for Gorenstein $\mathcal{P}_\lambda^m$, but  these results follow rather quickly.
Several of the Gorenstein examples are hypersimplices (Corollary~\ref{cor:hypersimplices}) which are known to have combinatorial formulas \cite{Kim-Hypersimplices} and are known to have regular unimodular triangulations \cite{LamPostnikov}. 
The remaining examples of Gorenstein $\mathcal{P}_\lambda^m$ are dilated standard simplices and arguments will follow in a similar vein to the case of $\lambda=(n)$.
\end{remark}

\section{Conclusion}
In this paper we study the Newton polytopes of Schur polynomials and show that they all have the integer decomposition property.
We determine which Schur polynomials have reflexive Newton polytopes,
and for which the Newton polytope is Gorenstein.
For the reflexive Newton polytopes of Schur polynomials,
we give the $h^\ast$-polynomials, and show that they are unimodal.

We also introduce a generalization of symmetric Grothendieck polynomials, called inflated symmetric Grothendieck polynomials.
We show that all these polynomials have saturated Newton polytope and their Newton polytopes have the integer decomposition property.
We characterize the partitions whose inflated symmetric Grothendieck polynomials have reflexive Newton polytopes and provide a table of their $h^\ast$-vectors.

The study of polynomials with saturated Newton polytope both introduces additional lattice polytopes of combinatorial interest and provides a new tool for approaching problems in Ehrhart theory. 
It may be fruitful to consider questions of the integer decomposition and reflexive polytopes for the Newton polytopes of other families of polynomials which are known to have SNP, such as chromatic symmetric polynomials and Schubert polynomials. 
Additionally, it is worth noting that the Newton polytopes of many polynomials of interest in 
algebraic combinatorics appear not only to have SNP, but also to have nice Ehrhart theoretic properties such as IDP and $h^\ast$-unimodality, either by our theorems or computationally. 
Perhaps it would be of interest to investigate if these Newton polytopes have other sought after properties, such as Ehrhart positivity. 


\appendix
\section{Tables of $h^\ast$-vectors for reflexive polytopes}

We present some tables of $h^\ast$-vectors for reflexive Newton polytopes arising from inflated symmetric Grothendieck polynomials $G_{h,\lambda}(\x)$ and Schur polynomials $s_\lambda(\x)$.

\begin{center}
	\begin{table}[h]
\small
	   \begin{tabular}{ | c | c | l |}
 	   \hline
$m$ & $h$ & $h^\ast$-vector\\ \hline
3 &	$\geq 2$ & 1, 31, 31, 1\\
4 &	$\geq 2$ & 1, 121, 381, 121, 1\\
5 &	$\geq 2$ & 1, 456, 3431, 3431, 456, 1\\
6 &	$\geq 2$ & 1, 1709, 26769, 60691, 26769, 1709, 1\\
7 &	$\geq 2$ & 1, 6427, 193705, 848443, 848443, 193705, 6427, 1\\
8 &	$\geq 2$ & 1, 24301, 1343521, 10350421, 19610233, 10350421, 1343521, 24301, 1\\ \hline
 	   \end{tabular}
       \caption{\label{tab:h*-vecs3} $h^\ast$-vectors of some reflexive Newton polytopes arising from $G_{h,\lambda}$ for 
${\lambda=((m+1)^{m-1},0)}$. For $m\geq 3$ and $h\geq 2$, $h(m-1) \geq 2(m-1) \geq (m+1)$. Thus, $\Newt(G_{2,\lambda}(\x)) = \Newt(G_{h,\lambda}(\x))$ for all $h\geq 2$. 
We should note that these polytopes are equivalent to the $\lambda$-permutatohedra for $\lambda=(n)$.
Hence, these are the same coefficients given by Proposition~\ref{prop:(n)_h*} and by OEIS A108267 \cite{oeis}
       }
	\end{table}
\end{center}

\begin{center}
	\begin{table}[h]
\small
	   \begin{tabular}{ | c | c | l |}
 	   \hline
$m$ & $h$ & $h^\ast$-vector\\ \hline
3 & 1&	1, 12, 12, 1\\
3 &	$\geq 2$ &	1, 19, 19, 1\\
4 &	$\geq 1$ & 1, 61, 183, 61, 1\\
5 &	$\geq 1$ & 1, 186, 1301, 1301, 186, 1\\
6 &	$\geq 1$ & 1, 554, 7974, 17756, 7974, 554, 1\\
7 &	$\geq 1$ & 1, 1639, 44997, 191955, 191955, 44997, 1639, 1\\
8 &	$\geq 1$ & 1, 4841, 241861, 1805773, 3393515, 1805773, 241861, 4841, 1\\
9 &	$\geq 1$ & 1, 14308, 1261900, 15539722, 49625029, 49625029, 15539722, 1261900, 14308, 1\\ \hline
 	   \end{tabular}
       \caption{\label{tab:h*-vecs2} $h^\ast$-vectors of some reflexive Newton polytopes arising from $G_{h,\lambda}$ for $\lambda=(2^\ell,0^{\ell+1})$ or $(2^{\ell-1},1,0^\ell)$. Since $\lambda_1=2$ in these cases, then for $m=3$ and all $h\geq2$, and for $m\geq4$ and all $h\geq1$, the Newton polytope $\Newt(G_{h,\lambda})$ stabilizes and is the truncation of the $m$-cube $[0,2]^m$ by the hyperplane $x_1+\cdots+x_m=m-1$.       
       }
	\end{table}
\end{center}

\begin{center}
	\begin{table}
	   \begin{tabular}{ | c | c | l p{0cm} |}
 	   \hline
	    $\lambda$ & n & $h^\ast$-vector &\\ \hline
		(21)	&	3&	1, 4, 1&\\
		(3)	&	3&	1, 7, 1&\\
		(211)	&	4&	1, 9, 9, 1&\\
		(22)	&	4&	1, 15, 15, 1&\\
		(4)	&	4&	1, 31, 31, 1&\\
		(2111)	&	5&	1, 16, 36, 16, 1&\\
		(221)	&	5&	1, 46, 136, 46, 1&\\
		(5)	&	5&	1, 121, 381, 121, 1&\\
		(21111)&	6&	1, 25, 100, 100, 25,1 &\\
		(222)	&	6&	1, 135, 920, 920, 135, 1&\\
		(6)	&	6&	1, 456, 3431, 3431, 456, 1&\\
		(211111) &	7&	1, 36, 225, 400, 225, 36, 1&\\
		(2221)	  &	7&	1, 386, 5405, 11964, 5405, 386, 1&\\
		(2111111) &	8&	1, 49, 441, 1225, 1225, 441, 49, 1&\\
		(2222)	   &	8&	1, 1099, 29337, 124187, 124187, 29337, 1099, 1&\\
		(22221)  &	9&	1, 3130, 152110, 1126258, 2112016, 1126258, 152110, 3130, 1&\\ \hline
 	   \end{tabular}
       \caption{\label{tab:h*-vecs} $h^\ast$-vectors of some reflexive Newton polytopes arising from Schur polynomials $s_\lambda$.}
	\end{table}
\end{center}

\pagebreak


\end{document}